\documentclass[12pt,reqno]{amsart}%

\usepackage{amsmath,amsthm}
\usepackage{amsfonts}
\usepackage{amssymb}
\usepackage{mathabx}
\usepackage{graphicx}
\usepackage{float}
\usepackage{hyperref}
\usepackage{fancyhdr}
\usepackage{enumerate}

\usepackage{amsmath}%
\newcommand{\sub}{\subseteq}
\newcommand{\eps}{\varepsilon}

\newcommand{\Aut}{\operatorname{Aut}}

\newcommand{\Ker}{\operatorname{Ker}}

\newcommand{\TF}{\operatorname{TF}}

\newcommand{\Id}{\operatorname{Id}}

\newcommand{\diag}{\operatorname{diag}}

\newcommand{\set}[2]{\left\{\left.#1\;\right|\; #2\right\}}
\newcommand{\rest}[2]{#1\raisebox{-0.5ex}{\mbox{$\mid_{#2}$}}}

\theoremstyle{plain}
\newtheorem{theorem}{Theorem}
\newtheorem{lemma}[theorem]{Lemma}
\newtheorem{proposition}[theorem]{Proposition}

\newtheorem{corollary}[theorem]{Corollary}
\theoremstyle{definition}
\newtheorem{definition}[theorem]{Definition}

\theoremstyle{remark}
\newtheorem{remark}[theorem]{Remark}

\newtheorem{example}[theorem]{Example}

\numberwithin{equation}{section}
\numberwithin{theorem}{section}

\begin{document}
\title{\textbf{Extensions of special 3-fields}}
\author{{Steven Duplij}}
\address{Center for Information Technology (WWU IT),
Westf\"{a}lische Wilhelms-Universit\"{a}t M\"{u}nster,
R\"ontgenstra\ss e 7-13, D-48149 M\"{u}nster, Germany}
\email{\small douplii@uni-muenster.de,
sduplij@gmail.com}

\author{{Wend Werner}}
\address{Mathematisches Institut, Westf\"{a}lische Wilhelms-Universit\"{a}t M\"{u}nster\\
Einsteinstrasse 62, D-48149 M\"{u}nster, Germany}

\begin{abstract}
We investigate finite field extensions of the unital 3-field, consisting of the unit element alone,
and find considerable differences to classical field theory.
Furthermore, the structure of their automorphism groups is clarified and the respective subfields
are determined.
In an attempt to better understand the structure of 3-fields that show up here we look at ways in which new
unital 3-fields can be obtained from known ones in terms of product structures, one of them the
Cartesian product which has no analogue for binary fields.
\end{abstract}
\maketitle

\markboth{}{}

\pagestyle{headings}

\thispagestyle{empty}

\section{Introduction}
Algebraic structure which is based on composing more than two elements can be traced back to
early work of D\"ornte \cite{dor3} and
Post \cite{pos} and has later shown an increase in interest, see for instance \cite{boh/wer2,cel,cro1,cro/tim,duplij2022,elg/bre,hig,kur1,lee/but},
and especially in physical model building \cite{abl2022,bag/lam1,azc/izq,nam0,ker2000,ker2012}. On many occasions, such a theory
substantially profit from
embedding objects into a larger structure in which such seemingly unconventional algebraic structure can be
reduced to more conventional concepts. For a typical example see the brief discussion of
ternary commutative groups in \cite{dup/wer2021}.

In cases where multiple algebraic operations of this kind are present (for example, the 3-fields investigated here) such an approach, however, is less successful, and the theory requires rather novel techniques. Our principal definition in the following is

\begin{definition}
A set $R$ equipped with two operations $R^3\to R$, called \emph{ternary addition} and \emph{ternary multiplication},
is a called a \emph{3-ring}, iff for each $r,r_{1,2,3},s_{1,2}\in R$
\begin{enumerate}
  \item there are additive and, respectively, multiplicative \emph{querelements} $\overline{r}\in R$
  and $\widehat{r}$ so that $r\hat{+}\overline{r}\hat{+}r_1=r_1$ as well as $r\widehat{r}r_1=r_1$,
  \item $s_1s_2(r_1\hat{+}r_2\hat{+}r_3)=s_1s_2r_1\hat{+}s_1s_2r_2\hat{+}s_1s_2r_3$, and
  \item both operations are associative (i.e.\ no brackets are needed for multiple applications of these
  operations)
 \end{enumerate}
$R$ is called \emph{commutative}, if the order of factors and summands can be permuted in any possible way,
and \emph{unital} iff there is an element $1\in R$ with $11r=r$ for all $r\in R$.

A (unital) \emph{3-field} $K$ is a (unital) 3-ring iff for each $k\in K$ there is $\widehat{k}\in K$
so that $k\widehat{k}k_1=k_1$ for all $k_1\in K$.
\end{definition}
In a 3-ring (or, more generally, in a commutative 3-group), the presence of a multiplicative unit alone
allows to introduce a binary product of commutative groups
\begin{equation}
x\cdot y=x1y,
\end{equation}
which has the property that applying it twice on three factors results in the original ternary product.
Similarly, for a 3-ring $R$, a zero element $0$
is defined by requiring that
\begin{equation}
0xy=x0y=xy0=0\qquad\text{for all $x,y\in R$}.
\end{equation}
Such an element is uniquely determined
and as in the case of a 1-element for the multiplication, it allows for reducing ternary addition of $R$
to a binary one.

The fields we will investigate in the following will come equipped with a ternary addition (and no zero element),
and a multiplication that possesses a unit so that, right from the outset, we will assume multiplication to be binary. We then have
\begin{theorem}[\cite{dup/wer2021} Theorems 3.4, 6.1]
A unital 3-field $F$ embeds into a binary field $K$ (with its inherited ternary addition and binary multiplication)
if for each $y\in F\setminus\{1\}$ the equation
\begin{equation}
x+y-xy=1
\end{equation}
has the only solution $x=1$. Whenever $F$ is finite this happens if and only iff $F=\{1\}$.
\end{theorem}

Let us explain some of the basic features of the theory that will be used in the following (details are in \cite{dup/wer2021}).

To each unital 3-field $F$ belongs a uniquely determined (binary) local ring $\mathcal{U}(F)$ into which it embeds as the
subset $\mathcal{U}(F)^\ast$ of units in such a way that the ternary sum of $F$ coincides with the binary sum of $\mathcal{U}(F)$ applied twice. (And, conversely, the units of any local ring $\mathcal{U}$ with residual
field $\mathbb{F}_2$ form a unital 3-field, with its inherited structure.)
We have $\mathcal{U}(F)=\mathcal{Q}(F)\cup F$,
where the binary non-unital ring $\mathcal{Q}(F)$ consists of all mappings (pairs) $q_{a,b}: f\longmapsto f\hat{+}a\hat{+}b$,
$a,b\in F$, with addition and multiplication coming from pointwise operations on the set of mappings $F\to F$.
We will frequently make use of the fact that each $q\in\mathcal{Q}(F)$ has a unique representation in the form
$q=q_{1,f}$, $f\in F$.
As we will consider every unital 3-field $F$ to be naturally embedded into the binary ring $\mathcal{U}(F)$
so that the querelement $\overline{f}$ from $F$ becomes $-f$.

\begin{remark}
It is for this reason that we no longer will formally distinguish between the binary and ternary sums and
write + throughout.
\end{remark}

Any ternary morphism $\phi:F_1\to F_2$ canonically extends to a binary morphism $\mathcal{Q}(\phi):\mathcal{Q}(F_1)\to\mathcal{Q}(F_2)$
(and hence to a binary morphism $\mathcal{U}(\phi):\mathcal{U}(F_1)\to\mathcal{U}(F_2)$) via $\phi(q_{a,b})=q_{\phi(a),\phi(b)}$. The kernel
of $\mathcal{Q}(\phi)$ is a binary ideal in $\mathcal{U}(F_1)$, and we will address these binary ideals of $\mathcal{U}(F_1)$
as \emph{the (ternary) ideals of $F_1$}. We also will formally write a short exact sequence of unital 3-fields as
\begin{equation}
0\longrightarrow J\longrightarrow F\longrightarrow F_0\longrightarrow 0,
\end{equation}
with the understanding that we actually are dealing with unital 3-fields $F_1,F$ and $F_0$,
a short exact sequence of binary rings,
\begin{equation}
0\longrightarrow\mathcal{U}(F_1)\longrightarrow\mathcal{U}(F)\longrightarrow\mathcal{U}(F_0)\longrightarrow 0,
\end{equation}
and the ideal $J\sub\mathcal{Q}(F_1)$ arising as the kernel of the second arrow.

Using analogous definitions for unital 3-rings, the quotient of a unital 3-ring
by an ideal $J$ of $\mathcal{U}(R)$ is a unital 3-field iff for any proper ideal $J_0$ of $\mathcal{U}(R)$ for which
$J\sub J_0$ it follows that $J_0\cap R=\emptyset$. Note that this condition is automatically satisfied when $R$ already
is a 3-field.

The basic examples are the prime fields $\TF(n)=\set{2k-1}{k=1,\ldots,2^{n-1}}\sub\mathbb{Z}/2^n$ and
$\TF(\infty)$, the ternary field of quotients of the unital 3-ring $2\mathbb{Z}+1$. As unital 3-fields, they are
generated by the element 1. Since the unit of each unital 3-field $F$
generates a uniquely determined prime field $P_n$ inside $F$, the number of elements in this subfield,
the characteristic $\chi(F)=2^{n-1}$ of $F$, is well-defined.

Extensions of 3-fields are a much more complicated subject than its binary counterpart. The present investigation
is a first attempt at a deeper understanding. The next section deals with products which exist
thanks to the absence of a zero element. In some cases, they provide 3-field extensions. The final section is devoted to
3-field extensions of $\{1\}$, which are numerous and follow only in some cases the paths of Galois theory.

\section{Products of 3-fields}
\subsection{Cartesian Products}
Since unital 3-fields are supposed to be proper, for each pair of unital 3-fields $F_{1,2}$,
all elements of the Cartesian product $F_1\times F_2$ posses a multiplicative inverse so that
$F_1\times F_2$ is itself a unital 3-field,
under pointwise operations. For the same reason, however, there is no canonical embedding of one
of these fields into $F_1\times F_2$. (There is an exception in characteristic 1, though:
In this case
\begin{equation}
\iota_1:F_1\to F_1\times F_2,\qquad\iota_1(f)=(f,1_2)
\end{equation}
is an embedding with $\pi_1\iota_1=\Id_{F_1}$.)

\begin{theorem}\label{cart-product-char}
Let $F$ be a unital 3-field. Then the following are equivalent:
\begin{enumerate}
\item
$F$ is the Cartesian product of two 3-fields $F_{1,2}$.
\item
There are $Q_{1,2}\sub Q(F)$ so that as a binary ring, $\mathcal{Q}(F)=Q_1\oplus Q_2$.
\end{enumerate}
\end{theorem}
\begin{proof}
If $F=F_1\times F_2$ then $q_{(a_1,a_2),(b_1,b_2)}\in\mathcal{Q}(F_1\times F_2)$
acts by $q_{(a_1,a_2),(b_1,b_2)}(f_1,f_2)=(a_1+b_1+f_1,a_2+b_2+f_2)$
establishing the binary ring isomorphism
\begin{equation}
Z:\mathcal{Q}(F_1\times F_2)\to\mathcal{Q}(F_1)\oplus\mathcal{Q}(F_2),\qquad
q_{(a_1,a_2),(b_1,b_2)}\mapsto\left(q_{a_1,b_1},q_{a_2,b_2}\right).
\end{equation}
Assuming $\mathcal{Q}(F)=Q_1\oplus Q_2$ as rings, both $Q_i$ are ideals of $\mathcal{Q}(F)$
and hence yield 3-fields $F_i=F/Q_i$ with corresponding quotient maps $\pi_i$. Then $\pi(f)=(\pi_1(f),\pi_2(f))$
defines an injective morphism $F\to F_1\times F_2$ of 3-fields. Because $\pi_1^{-1}(f_1)=f^0_1+Q_1$ for some
$f_1^0\in F$ and $\pi_2(f^0_1+Q_1)=F_2$, $\pi$ is surjective.
\end{proof}

The above result and its proof can easily be extended to infinite products $\prod_{i\in I}F_i$ of
3-fields, and this is the product in the categorical sense: Whenever, for a 3-field $G$, there are
morphisms $\psi_i: G\to F_i$ there is a unique mapping $\psi:G\to\prod_{i\in I}F_i$ which can be shown
to be a 3-field morphism.

Note also that $\mathcal{Q}(F)=\mathcal{Q}(F_1)\oplus\mathcal{Q}(F_2)$ is not related to a similar
decomposition of $\mathcal{U}(F)$, since
\begin{equation}
\left[\mathcal{Q}(F_1)\oplus\mathcal{Q}(F_2)\right]\cup\left[F_1\times F_2\right]
\neq
\left[\mathcal{Q}(F_1)\cup F_1\right]\oplus\left[\mathcal{Q}(F_2)\cup F_2\right]
\end{equation}

Observe that modifying a construction
from \cite{dup/wer2021} in order to create direct sums is futile: Letting, for an
odd number of finite 3 fields $F_1,\ldots,F_n$ (odd, in order to have a unit and, possibly, invertibility),
\begin{equation}
\bigoplus_{i=1}^n F_i=\set{(f_i)\in\bigoplus_{i=1}^n\mathcal{U}(F_i)}{\sum_{i=1}^n\partial\left(f_i\right)=1}
\end{equation}
where $\partial:\mathcal{U}(F)\to \mathbb{F}_2$ denotes the natural grading for the unital 3-field $F$,
it turns out that $(f_i)$ possesses an inverse iff $(f_i)^N=(1,\ldots,1)$ and it follows that
the set of invertible elements within $\bigoplus_{i=1}^n F_i$ equals $\bigtimes_{i=1}^n F_i$.

\subsection{Semi-direct Products and Unitization of Algebras}
We will base the notion of a semi-direct products of 3-fields on the concept of
a split short exact sequence of unital 3-fields $0\rightarrow J\rightarrow F\rightarrow G\rightarrow 0$.

\begin{definition}
The unital 3-field $F$ is the \emph{(internal) semi-direct} product of the ideal $J\sub\mathcal{Q}(F)$
and the subfield $G$ iff $G$ is the image of an epimorphism $\pi:F\to G$ with $\Ker\pi=J$
and right inverse $\iota:G\to F$.
\end{definition}

More abstractly, semi-direct products are connected to algebras over 3-fields.
Recall \cite{dup/wer2021}

\begin{definition}
Let $A$ be a binary ring and $F$ a unital 3-field. We call $A$ a binary algebra over $F$,
iff $F$ acts on $A$ in such a way that, for all $f,f_{1,2,3}\in F$ and $a,a_{1,2}\in A$,
\begin{enumerate}
  \item $f(a_1+a_2)=fa_1+fa_2$
  \item $(f_1+f_2+f_3)a=f_1a+f_2a+f_3a$
  \item $(f_1f_2)a=f_1(f_2a)$
  \item $1_Fa=a$
\end{enumerate}
Similarly, we will call $A$ a 3-algebra over $F$ iff $A$ is a 3-ring (with addition coming from
an underlying commutative 3-group) equipped with a binary product so that, for all
$f,f_{1,2,3}\in F$ and $a,a_{1,2}\in A$,
\begin{equation}
f(a_1+a_2+a_3)=fa_1+fa_2+fa_3
\end{equation}
while axioms (2)-(4) for a binary algebra are left untouched.
\end{definition}

Note that we do not assume $A$ to be unital.

\begin{example}
$\mathcal{Q}(F)$ is an $F$-algebra where the action of $F$ is given by $f_1q_{1,f}=q_{f_1,f_1f}$,
whenever $q_{1,f}\in\mathcal{Q}(F)$ and $f_1\in F$.
\end{example}

\begin{example}\label{standard semi-direct}
More generally,
if the morphism of unital 3-fields $\pi:F\to G$ possesses the right inverse $\sigma:G\to F$,
the ideal $J=\Ker\mathcal{Q}(\pi)$ is a G-algebra with $G$-action $g.j=\sigma(g)j$.
\end{example}

We will also need the fact that a binary $F$-algebra $A$ becomes a binary module over the
local ring $\mathcal{U}(F)$ by letting
\begin{equation}
q_{1,f}a=a+fa,\quad f\in F,\ a\in A.
\end{equation}
In case $A$ is a 3-algebra, $\mathcal{Q}(F)$ can only act on $\mathcal{Q}(A)$ via
\begin{equation}
q_{1,f}q_{a,b}=q_{a,b+fa+fb}.
\end{equation}
Semi-direct products for unital 3-fields will be related to the unitization of an
algebra over a unital 3-field. There are two variants, one with purely binary
operations, and another one for which addition becomes ternary.

\begin{definition}
Let $A$ be an algebra over the unital 3-field $F$. Then the \emph{binary unitization} of $A$ is defined
on the additive direct sum
\begin{equation}
A_F^{++}=\mathcal{U}(F)\oplus A
\end{equation}
with multiplication
\begin{equation}
(u_1,a_1)(u_2,a_2)=(u_1u_2,u_1a_2+u_2a_1+a_1a_2).
\end{equation}
\emph{Ternary unitization} is based on the additively written commutative 3-group $A_F^+=F\oplus A$
and a binary product which is equal to the one in the binary case.
\end{definition}

It is straightforward to check that $A_F^{++}$ is a binary unital ring, and $A^+_F$ is a unital 3-ring.

\begin{theorem}\label{nilpotent case}
The unitization $A^+$ of a nilpotent algebra $A$ over the unital 3-field $F$ is a unital 3-field,
and it follows that $A^+$ is a semi-direct product with canonical quotient map $A^+\to F$ and natural
split $F\rightarrow F\oplus 0$.
\end{theorem}

\begin{proof}
Since $A^+$ always is a unital 3-ring we still have to show that each element $(f,a)$ has an inverse.
But, due to nilpotency of $A$,
\begin{equation}
(1,a)^{-1}=\left(1,\sum_{\nu=1}^N(-1)^\nu a^\nu\right)
\end{equation}
for $N$ large enough, and hence $(f,a)^{-1}=f^{-1}(1,f^{-1}a)^{-1}$.
\end{proof}

\begin{corollary}
For a finite unital 3-field $F_0$ a subfield $F$ of $F_0$, and an ideal $J\sub\mathcal{Q}(F_0)$ equipped with
the natural action of $F$, $A^+_F$ is a unital 3-field.
\end{corollary}

There are infinite unital 3-fields $F$ for which $Q(F)$ is not nilpotent so in order to find 3-fields
among unitizations we need to define invertibility before unitization.

\begin{definition}\label{hash-def}
An algebra $A$ over a unital 3-field $F$ is called a \emph{$\mathcal{Q}$-algebra} iff for each $a\in A$ there
is $a^\hash\in A$ such that $aa^\hash=a+a^\hash$.
\end{definition}

\begin{example}
The simplest example is an algebra $A$ over the 3-field $\TF(0)=\{1\}$. Equivalently,
$A$ is a binary ring such that $a+a=0$ for all $a\in A$ and, at the same time,
a (binary) algebra over $\operatorname{GF}(2)$. If also $a^2=0$ for all $a\in A$,
$A^+_{\TF(0)}$ is a unital 3-field in which $(1,a)^{-1}=(1,a)$ for all $a\in A$.
If the product of $A$ vanishes identically and if we select a basis $B$
for the vector space $A$, we find
\begin{equation}
A^+_{\TF(0)}\cong\TF(1)^{|B|}
\end{equation}
This example also covers the case in which $F$ acts trivially on $A$,
i.e.\ $fa=a$, for all $f\in F$, $a\in A$.
\end{example}

\begin{example}\label{hash-ex}
Each nilpotent algebra $A$ over the unital 3-field $F$ is a $\mathcal{Q}$-algebra with
\begin{equation}
a^\hash=-\sum_{\nu=1}^N a^\nu,
\end{equation}
where $a^{N+1}=0$.
\end{example}

\begin{lemma}\label{Q-unit}
An algebra $A$ over a unital 3-field $F$ is a $\mathcal{Q}$-algebra if and only if $A^+_F$ is a unital
3-field.
\end{lemma}

\begin{proof}
Suppose $A$ is a $\mathcal{Q}$-algebra. We must prove that each element is invertible, and, as in the proof
of Theorem~\ref{nilpotent case}, it suffices to show $(1,q)^{-1}$ exists. But it follows from the definition
of $q^\hash$ that $(1,q)(1,q^\hash)=(1,0)$.

Conversely, if $(1,q)$ has inverse $(f,\overline{q})$ it follows that $f=1$ and $q+\overline{q}+q\overline{q}=0$
so that $\overline{q}$ provides the $\hash$-element, providing $A$ with the structure of a $\mathcal{Q}$-algebra.
\end{proof}

\begin{theorem}
Let $F$ be a unital 3-field. There exists a bijective correspondence between the unitization of
$\mathcal{Q}$-algebras $A$ over $F$ and semi-direct products
of 3-fields $J\rtimes F$:
\begin{enumerate}
\item
For each $\mathcal{Q}$-algebra $A$ over $F$, the mapping
$\pi_{A,F}:A^+_F\to F$, $(f,a)\mapsto f$ establishes the short
exact sequence of unital 3-fields
\begin{equation}\label{Q-split}
0\longrightarrow A\longrightarrow A^+_F\longrightarrow F\longrightarrow 0
\end{equation}
which is split by $\sigma_{A,F}:f\mapsto (f,0)$ and so, $A^+_F=A\rtimes F$.
\item
Conversely, each split exact sequence $0\rightarrow J\rightarrow F_0\rightarrow F\rightarrow 0$
naturally defines the structure of a $\mathcal{Q}$-algebra over $F$ on $J$ and it follows
that $F_0=J\rtimes F$ is isomorphic to $J_F^+$.
\end{enumerate}
\end{theorem}

\begin{proof}
By the definition of $A_F^+$ and Lemma~\ref{Q-unit}, $\pi_{A,F}$ and $\sigma_{A,F}$ are
morphisms of unital 3-fields giving rise to the split exact sequence $0\to A\to A^+_F\to F\to 0$.

In order to prove the second statement, denote by $\sigma:F\to F_0$ the right inverse to the
quotient morphism $\pi:F_0\to F$. Example~\ref{standard semi-direct} shows $J$ is a
$\mathcal{Q}$-algebra over $F$, and the mapping
\begin{equation}
J_F^+\longrightarrow F_0,\qquad (f,q)\longmapsto \sigma(f)+q
\end{equation}
is an isomorphism of unital 3-fields.
\end{proof}

\begin{example}
Let $F$ be a subfield of $F_0$ and $J$ an ideal of $\mathcal{U}(F_0)$. For each $q=q_{1,f}\in J$,
$q^\hash=q_{1,f^{-1}}$ is in $J$ since $q^\hash=-qq^\hash-q$. It follows that $J$ is a $\mathcal{Q}$-algebra
over $F$ and so $F\ltimes J$ is a unital 3-field.
\end{example}

\begin{example}
Let $F_0$ be a unital 3-field with unital subfield $F_1$ and let $F=F_0\times F_1$. Then
$\mathcal{Q}(F)=\mathcal{Q}(F_0)\oplus\mathcal{Q}(F_1)$ and $\pi_2(f_0,f_1)\to f_1$
is a surjective morphism with kernel
\begin{equation}
J=\set{q=q_{(1,1),(g_0,g_1)}\in\mathcal{Q}(F)}{0=\mathcal{Q}(\pi_2)(q)=q_{1,f_1}}
=\mathcal{Q}(F_0)\oplus\{0\}
\end{equation}
$\pi_2$ has the left inverse $\sigma:F_1\to F$, $\sigma(f_1)=(f_1,f_1)$ which produces
an action of $\diag F_1\times F_1\cong F_1$
on $J$ given by
\begin{equation}
(f_1,f_1)q_{1,f_0}=f_1q_{1,f_0}=q_{1,f_1+f_1f_0-1}
\end{equation}
$J$ is a $\mathcal{Q}$-algebra over $F_1$ with $q_{1,f_0}^\hash=q_{1,f_0^{-1}}$, and the morphism
$\Phi:J_{F_1}^+\to F_0\times F_1$,
\begin{equation}
\Phi(f_1\oplus q_{1,f_0})=(f_1,f_1+q_{1,f_0})=(f_1,1+f_0+f_1)
\end{equation}
is an isomorphism.
\end{example}

\subsection{Creating a 3-field action}
We complement the above with an intrinsic characterization of the binary rings $\mathcal{Q}(F)$, this
time without using the action of a unital 3-field. First an observation: The ring structure of $\mathcal{Q}(F)$
uniquely determines the underlying unital 3-field.

\begin{proposition}
Fix a unital 3-field $F_{1,2}$ are unital 3-fields and suppose $\Psi:Q(F_1)\to Q(F_2)$ is a binary ring morphism.
Define $\Phi: F_1\to F_2$ through $\Psi(q_{1,f})=q_{1,\Phi(f)}$. Then $\Phi$ is a unital 3-field morphism which
is an automorphism whenever $\Psi$ is.
\end{proposition}

\begin{proof}
Using the involved definitions, one checks that $\Phi$ respects addition and multiplication of $F$. Furthermore,
$0=\Psi(q_{1,\overline{1}})=q_{1,\overline{\Phi(1)}}$,
and so $\Phi$ is unital. If $\Psi$ is an automorphism
and $\Psi^{-1}(q_{1,a})=q(1,\widehat{\Phi}(a))$ then, necessarily $\widehat{\Phi}=\Phi^{-1}$.
\end{proof}

\begin{definition}
A commutative (binary) ring $Q$ is called a $\mathcal{Q}$-ring iff
\begin{enumerate}
  \item There exists a \emph{2-unit} $\tau\in Q$ so that $\tau q=q+q$ for all $q\in Q$.
  \item For each $q\in Q$ exists a unique \emph{$\hash$-element} $q^\hash\in Q$ with $q^\hash q=q+q^\hash$.
\end{enumerate}
The morphisms between $\mathcal{Q}$-rings are those ring morphisms that map the
respective $\tau$-elements onto each other and respect the $\hash$-involution.
\end{definition}

We note some consequences of these axioms:
\begin{enumerate}
  \item A $\mathcal{Q}$-ring is never unital, since we would have $1^\hash=1+1^\hash$ and so $1=0$.
  \item $\tau^\hash=\tau$ and $\tau^n=2^{n-1}\tau$
  \item There might be more than one 2-element: two such elements $\tau_{1,2}$ have to satisfy
  $(\tau_1-\tau_2)x=0$ for all $x\in Q$.
  \item Similarly, $q^\hash_{1,2}\in Q$ are $\hash$-elements for $q\in Q$ iff
  $(q_1^\hash-q_2^\hash)(q-1)=0$, within the unitization of $Q$. As we will see below, this
  determines the element $q^\hash$ uniquely.
\end{enumerate}

\begin{theorem}\label{Q-ring char}
$Q$ is a $\mathcal{Q}$-ring, iff there is a unital 3-field $F$ such that $Q=\mathcal{Q}(F)$.
\end{theorem}
\begin{proof}
Suppose $Q=\mathcal{Q}(F)$. Then, letting $\tau=q_{1,1}$ and $q_{1,f}^\hash=q_{1,f^{-1}}$,
we have turned $\mathcal{Q}(F)$ into a $\mathcal{Q}$-ring.
Conversely, whenever $F$ is a $\mathcal{Q}$-ring,
define ternary addition $\hat{+}$ as well as binary multiplication $\hat{\times}$ by
\begin{equation}
f_1\hat{+}f_2\hat{+}f_3=f_1+f_2+f_3-\tau,\qquad
f\hat{\times}g=\tau-f-g+fg
\end{equation}
Then the querelement of $f\in F$ for $\hat{+}$ is $\overline{f}=\tau-f$, the distributive law holds since
\begin{equation}
f\hat{\times}(f_1\hat{+}f_2\hat{+}f_3)=2\tau-3f-f_1-f_2-f_3+ff_1+ff_2+ff_3=
f\hat{\times}f_1\hat{+}f\hat{\times}f_2\hat{+}f\hat{\times}f_3,
\end{equation}
also $\tau\hat{\times}f=\overline{\tau+f-\tau f}=\overline{\overline{f}}=f$, and $f^\hash\hat{\times}f=
\overline{f+f^\hash-ff^\hash}=\overline{0}=\tau$.
\end{proof}

\begin{corollary}\label{nilpot-char}
For a nilpotent ring $Q$ there exists a unital 3-field $F$ with $Q=\mathcal{Q}(F)$ iff $Q$ contains a
2-unit $\tau$.
\end{corollary}

\begin{example}
In the simplest case, when the product on a binary ring $R$ of characteristic 2 vanishes, each element
$\tau\in R$ can serve as a 2-element, while $r^\hash=-r$. For the resulting unital 3-field $F(R)_\tau$
we have
\begin{equation}
r\hat{+}s\hat{+}t=r+s+t+\tau\qquad r\hat{\times}s=r+s+\tau,
\end{equation}
and whatever the choice of $\tau\in R$, the mapping $\Phi_\tau: r\longmapsto r+\tau$ provides an isomorphism
between $F(R)_0$ and $F(R)_\tau$.
\end{example}

\begin{example}
The rings of pairs $Q(n)=\set{2k\in\mathbb{Z}/2^n}{k=0,\ldots,2^{n-1}-1}$ for the
3-fields $TF(n)=\set{2k-1\in\mathbb{Z}/2^n}{k=1,\ldots,2^{n-1}}$ are $\mathcal{Q}$-rings,
with $\tau=2$ and, taking the inverse inside $\mathbb{Z}/2^n$,
\begin{equation}
q^\hash=\frac{q}{q-1}.
\end{equation}
Similarly, $Q(\infty)=\set{p/q}{p=2r,\ q=2s+1,\ r,s\in\mathbb{Z}}$ is a $\mathcal{Q}$-ring with
$\tau=2$ and
\begin{equation}
(p/q)^\hash=\frac{p}{p-q}.
\end{equation}
Note that in these examples, $\tau$ is unique.
\end{example}

\begin{example}
If $A$ is a $\mathcal{Q}$-algebra over the unital 3-field $F$, the ideal $\mathcal{Q}(F)\oplus A$ of the
binary unitization $A_F^{++}$ is (still a $\mathcal{Q}$-algebra over $F$ but also) a $\mathcal{Q}$-ring,
with $\tau=(q_{1,1},0)$. The map
\begin{equation}
\Psi:\mathcal{Q}(A_F^+)\longrightarrow \mathcal{Q}(F)\oplus A,\qquad
q_{(1,0),(f,a)}\longmapsto\left(q_{1,f},a\right),
\end{equation}
establishes an isomorphism.
\end{example}

\section{3-field extensions of $\{1\}$}\label{galois section}

It could be argued that an extension of a 3-field $F_0$ should be any unital 3-field $F$
arising in a short exact sequence
\begin{equation}
0\longrightarrow J\longrightarrow F\longrightarrow F_0\longrightarrow 0,
\end{equation}
where $J$ is (identified with) an ideal of $\mathcal{U}(F)$. In this case, $F=F_0\times F_1$
would qualify as an extension of $F_0$, and the projection onto $F_0$ leads to the
short exact sequence
\begin{equation}
0\longrightarrow \mathcal{Q}(F_1)\times 0\longrightarrow F_0\times F_1\longrightarrow F_0\longrightarrow 0.
\end{equation}
We will nonetheless follow the established notion
and call the unital 3-field $F$ an extension of the unital 3-field $F_0$ in case $F_0$
embeds into $F$.

In the present situation the structure of all subfields is more complex: Since $\{1\}$
being a subfield of a unital 3-field $F$ is equivalent to $\chi(F)=1$, these extensions coincide
with all unital 3-fields of characteristic 1. The extensions we will consider here belong to the
following class.

\begin{definition}
Fix a unital 3-field $F$ as well as natural numbers $n_1,\ldots,n_k\in\mathbb{N}$ and put
\begin{equation}
F(n_1,\ldots,n_k)=\set{1+\sum_{0\neq\alpha}\eps_\alpha(x-1)^\alpha}{\eps_\alpha\in\mathbb{F}_2,\ (x-1)^{n_\kappa}=0,\ \kappa=1,\ldots,k}.
\end{equation}
\end{definition}

Note that by using semi-direct sums we may reduce the number of variables in this example:
Consider the map $x_i\longmapsto 1$ and extend it to an epimorphism $\pi_i$ on $F(n_1,\ldots,n_k)$,
\begin{equation}
1+\sum_{0\neq\alpha}\eps_\alpha(x-1)^\alpha\longmapsto
1+\sum_{0\neq\alpha,\ \alpha_i=0}\eps_\alpha(x-1)^\alpha
\end{equation}
with image $F(n_1,\ldots,n_{i-1},n_{i+1},\ldots n_k)$ and kernel consisting of the ideal
\begin{equation}
\mathfrak{J}_i=(x_i-1)\sum_\alpha\eps_\alpha(x-1)^\alpha.
\end{equation}
The elements of $F(n_1,\ldots,n_{i-1},n_{i+1},\ldots n_k)$ embed naturally into $F(n_1,\ldots,n_k)$
and act on $\mathfrak{J}_i$ by multiplication so that
\begin{equation}
F(n_1,\ldots,n_k)=F(n_1,\ldots,n_{i-1},n_{i+1},\ldots n_k)\ltimes\mathfrak{J}_i
\end{equation}

A slightly more sophisticated way of writing down these 3-field extensions is obtained in the following way.
\begin{definition}
Fix a finite Abelian binary group $G$ as well
as a local (binary) ring $R$ with residual field $\mathbb{F}_2=\{0,1\}$. Consider the group algebra $RG$
over $R$. Then
\begin{equation}
F_G=\set{f\in RG}{f(0)\in R^\ast}
\end{equation}
is called the \emph{ternary group algebra of $G$ over the 3-field $R^\ast$}.
\end{definition}
\begin{theorem}
$F_G$ is a 3-field, extending $R^\ast$, and, in case $R=\{0,\ldots,2^n-1\}=\mathcal{U}(1,3,\ldots,2^n-1)$,
each finite unital field extensions of the unital 3-field $\{1,3,\ldots,2^n-1\}$ which is contained in
an extension generated by a single elements is isomorphic to one of these fields.
\end{theorem}

\subsection{Extensions of characteristic 0, generated by a single element}
We are going to consider extensions of $F_0=\TF(0)=\{1\}$, with prime field identical to $\TF(0)$,
and generated by a single element.

\begin{theorem}\label{singly-gen-char}
Any finite unital 3-field $F$ of characteristic 1 which is generated by a single element $x$ is isomorphic
to $F_0(n_P)$ where $n_P$ is the smallest natural number with $(1-x)^n=0$. If $\jmath(P)=\min\set{i}{\eta_i\neq 0}$ then
$n_P=\lceil n/k\rceil$.
\end{theorem}

\begin{proof}
Consider the polynomial 3-ring $F_0[x]=\set{1+\sum_{\nu=1}^n\eps_\nu(x-1)^\nu}{\eps_\nu\in\mathbb{F}_2}$
for which $\mathcal{U}F_0[x]$ is the principal domain $\mathbb{F}_2[x]$. Consequently, whenever $F$ is
generated by a single element, there is a polynomial $P\in\mathcal{Q}F_0[x]=
\set{\sum_{\nu=0}^n\eps_\nu(x-1)^\nu}{\eps_\nu\in\mathbb{F}_2}$ such that $F=F_0[x]/\langle P\rangle$.
The polynomial $P$ cannot be of the form
\begin{equation}
P=(x-1)^n\left(1+(x-1)^k+\sum_{\nu=k+1}^N\pi_\nu(x-1)^\nu\right),\qquad \pi_\nu\in\mathbb{F}_2,
\end{equation}
because then $\langle 1+(x-1)^k+\sum_{\nu=k+1}^N\pi_\nu(x-1)^\nu\rangle$ would be strictly larger than $\langle P\rangle$,
intersect $F_0[x]$ and thus contradict \cite{dup/wer2021}[Theorem 1]. So, $P=(x-1)^n$ for the smallest $n$ with $(x-1)^n=0$ within
the 3-field $F$.
\end{proof}

Any attempt at creating a theory related to classical Galois theory has to deal with a number
of obstacles, one of them, for example, the less useful factorization of polynomials: In
$F_0(4)$, for example, the polynomial $1+(x-1)+(x-1)^3=\left[1+(x-1)\right]\left[1+(x-1)^3\right]$
vanishes at 1 when using the right hand representation, while it takes the value 1 when 1 is plugged into
$1+(x-1)+(x-1)^3$.

\begin{proposition}\label{ideals in F_0(n)}
The ideals of $F_0(n)$ are
\begin{equation}
\mathfrak{I}_k=\set{\sum_{\nu\geq k}\eps_\nu(x-1)^\nu}{\eps_\nu\in\mathbb{F}_2},
\qquad k=1,\ldots,n-1
\end{equation}
Consequently, $F_0(n)$ never is a Cartesian product of 3-fields.
\end{proposition}

\begin{proof}
Since the ring $\mathcal{U}F_0(n)=\set{\sum_{\nu=0}^{n-1}\eps_\nu(x-1)^\nu}{\eps_i\in\mathbb{F}_2}$ is principal,
for each proper ideal $\mathfrak{I}$ in $\mathcal{Q}F_0(n)$ the ideal
\begin{equation}
<\mathfrak{I},F_0(n)>=\set{\sum q_i(1+r_i)}{q_i\in\mathfrak{I},\ r_i\in\mathcal{Q}F_0(n)}
\sub\mathfrak{I}+\mathfrak{I}
\end{equation}
is the same as $\mathfrak{I}$ and so $\mathfrak{I}=<P_0>$, with $P_0\in\mathcal{Q}F_0(n)$.
Since each $P\in\mathcal{Q}F_0(n)$ can be written $P=(x-1)^s(1+R)$, $(1+R)$ invertible, $1\leq s\leq n-1$,
it follows that the ideals of $\mathcal{Q}F_0(n)$ are precisely those of the form $\mathfrak{I}_k=<(x-1)^k>$.
But $\mathfrak{I}_s\cap\mathfrak{I}_t=\mathfrak{I}_{\max\{s,t\}}$,
$\mathfrak{I}_s+\mathfrak{I}_t=\mathfrak{I}_{\min\{s,t\}}$ and, by applying Theorem~\ref{cart-product-char},
we conclude that $F_0(n)$ never is a proper Cartesian product of 3-fields.
\end{proof}

Whether or not $F_0(n)$ can be written as a semi-direct product is a slightly more delicate question, and
we will return to it elsewhere.

\begin{lemma}\label{morphisms}
Denote by $\varphi_n:F_0(n)\to F_0(n)$ the Frobenius morphism $P\longmapsto P^2$, by $F_0(n)^2$
its image, and, for $k\geq 2$ let $\mu_{n,k}:F_0(n)\to F_0(k)$ be defined by
\begin{equation}
\mu_{n,k}\left(1+\sum_{i=1}^{n-1}\eps_i(x-1)^i\right)=1+\sum_{i=1}^{k-1}\eps_i(x-1)^i.
\end{equation}
\begin{enumerate}
  \item
  $\varphi_n$ and $\mu_{n,k}$ are morphisms,
\begin{equation}
  \Ker\varphi_n=\mathfrak{I}_{\lceil n/2\rceil},\qquad\text{and}\qquad\Ker\mu_{n,k}=\mathfrak{I}_k
\end{equation}
  \item
  The product on $\Ker\varphi_n$ vanishes identically so that $(1+P_1)(1+P_2)=1+P_1+P_2$.
\end{enumerate}
\end{lemma}

\begin{proof}
\textbf{(1)}
Both maps are well-known (and easily seen to be) morphisms. Furthermore, the image of $\varrho_n$ can
naturally be embedded into $F_0(n)$.
\end{proof}

The additive structure of $F_0(n)$ is most transparent on $\mathcal{U}F_0(n)$, which is a (binary) vector
space with basis $\set{(x-1)^k}{k=0,\ldots,n-1}$. The multiplicative structure of these fields is the following.

\begin{theorem}\label{generate_mult}
Each element $f\in F_0(n)$ has a uniquely defined factorization $f=\gamma_0^{\alpha_0}\ldots\gamma_K^{\alpha_K}$
where
\begin{equation}
\gamma_k=1+(x-1)^{2k+1}\qquad 0\leq 2k+1\leq n-1,
\end{equation}
and it follows that the multiplicative group underlying $F_0(n)$ is isomorphic to the direct product
$C_{n,0}\times\ldots\times C_{n,K_n}$ of the cycles $C_{n,k}$ of the $\gamma_k$, with $K_n=\max\set{k}{2k+1\leq n-1}$,
and $C_{n,k}\cong\mathbb{Z}/2^{s(n,k)}$, $s(n,k)=\min\set{2^t}{2^tk\geq n}$. Consequently, with respect to its
multiplicative structure,
\begin{equation}
F_0(n)\cong\prod_{0\leq 2k+1\leq n}\left(\mathbb{Z}/2^{s(n,k)}\right)^{r(n,k)},\qquad
\end{equation}
\end{theorem}

\begin{proof}
The cases $n=2$ is trivial, and the elements of $F_0(3)$ are $\gamma_0^k$, $k=0,\ldots,3$.
If the statement is true for $F_0(n)$, $f\in F_0(n+1)$ has a uniquely defined factorization
$f=g\gamma_0^{\alpha_0}\ldots\gamma_K^{\alpha_K}$ where $g\in\set{1+\eps(x-1)^n}{\eps=0,1}$, the multiplicative
kernel of the canonical morphism $\pi:F_0(n+1)\to F_0(n)$. If $n$ is odd, this already proves the claim;
in case $n$ is odd and $g\neq 1$ we must have $g=(1+(x-1)^\ell)^{2^m}$, $\ell$ odd, and the statement
of the result follows for $n+1$.
\end{proof}

\subsection{Intermediate Fields}

We start with slightly generalizing Corollary~\ref{nilpot-char}.
\begin{theorem} Suppose $Q$ is a nilpotent ring with 2-unit $\tau$ and let $F$ be the
unital 3-field with $\mathcal{Q}(F)=Q$.
\begin{enumerate}
  \item
  $Q'\rightsquigarrow 1+Q'\sub F$ establishes a one-to-one correspondence between the subrings $Q'$ of $Q$
  with $\tau\in Q'$ and the unital 3-subfields $F'$ of $F$.
  \item
  The unital 3-subfield $F_0$ is generated by elements $1+q_1,\ldots,1+q_n$ iff $\mathcal{Q}(F_0)$
  is generated by $\tau=q_{1,1}$ and $q_1,\ldots,q_n$.
  \item
  The automorphisms $\Phi$ of $F$ leaving the subfield $F'$ pointwise fixed correspond to the automorphisms
  $\mathcal{Q}\Phi$, leaving $Q'=\mathcal{Q}(F')$ pointwise fixed.
\end{enumerate}
\end{theorem}

\begin{proof}
\textbf{(1)} If $F_0$ is a unital 3-subfield of $F$, the unit of $F_0$ must equal the one in $F$, and
$\mathcal{Q}(F_0)$ is a subring containing $\zeta=q_{1,1}$ which has the property that
$\zeta q=q+q$ for all $q\in\mathcal{Q}(F_0)$.
Note that $1+\mathcal{Q}(F_0)=F_0$.

Conversely, assume that $Q\sub\mathcal{Q}(F)$ with $\tau\in Q$ and
define $F_1=\set{f\in F}{q_{1,f}\in Q}$. We have $1\in F_1$ since $\tau\in Q$
and $-1\in F_1$ because $0\in Q$. It follows similarly that for all $f,g\in F_1$,
\begin{equation}
1+f+g\in F_1,\qquad f+g+gf\in F_1,
\end{equation}
and that there is $\breve{f}$ with $1+f+\breve{f}=-1$. From this, and the nilpotency of $Q$ it follows
that $F_1$ is unital 3-field such that $F_1=1+Q$.

\textbf{(2)}
The subring $Q_0$ generated by $\tau$ and $q_1,\ldots,q_n$ consists of the elements $k\tau+\sum_\alpha\eps_\alpha q^\alpha$,
$k\in\mathbb{N}_0$. Note that due to nilpotency, this ring is of finite characteristic. Then $1+Q_0$ is a unital 3-field,
consisting of elements $k+\sum_\alpha\eps_\alpha q^\alpha$, $k$ odd, and resolving the brackets in
$k+\sum_\alpha\eps_\alpha (1+q)^\alpha$ shows that $1+Q_0$ is the unital 3-field generated by 1 and $1+q_1,\ldots,1+q_n$.
Conversely, if $F_0$ is generated by $1,1+q_1,\ldots,1+q_n$, $\mathcal{Q}(F_0)$ contains $\tau,q_1,\ldots,q_n$ and thus
the binary ring $Q_0$, generated by these elements. By the first part of this proof, $F_0=1+Q_0$, and so $Q_0=\mathcal{Q}(F_0)$.

\textbf{(3)} This follows from the involved definitions.
\end{proof}

\begin{remark}
The last theorem can be easily generalized to $\mathcal{Q}$-rings and their $\mathcal{Q}$-subrings, where
a $\mathcal{Q}$-subring $Q_1$ of the $\mathcal{Q}$-ring $Q$ is defined by the requirement that $\tau\in Q_1$
and that $Q_1$ is invariant under the $\hash$-operation.
\end{remark}

\begin{corollary}
For each ideal $J$ of the finite unital 3-field $F$ the subalgebra $1+J$ a unital subfield.
\end{corollary}
\begin{proof}
For $\zeta=q_{1,1}\in\mathcal{Q}(F)$ we have
$\zeta=q(q-1)^{-1}\in J$.
\end{proof}

\begin{corollary}
For a unital 3-subfield $F$ of $F_0(n)$ which is generated by polynomials $P_1,\ldots,P_g$ there exist natural
numbers $n_1,\ldots,n_g$ such that $F\cong F_0(n_1,\ldots,n_g)$.
\end{corollary}

\begin{example}
The smallest subfields, those of order 2, are generated by polynomials
\begin{equation}
1+(x-1)^k(1+P_1)=\left(1+(x-1)^k\right)(1+P_1)-P_1
\end{equation}
with $2k\geq n$.
We call $k$ the \emph{lower degree} of the subfield.
As we will see, the automorphisms which are constant on such a field are of the form $\Psi_P$ with
$P_0=(x-1)+(x-1)^\ell(1+P_1)$ with $\ell\geq n-k+1$ so that $\Aut F_0(n)$ cannot distinguish between
subfields of the same degree in terms of fixed point subgroups. On the other hand, all elements of
$\Aut F_0(n)$ preserve $k$, and, since
\begin{equation}
\Psi(1+P_1)=\left(\Psi(P)-1\right)\Psi(x-1)^{-k}-1=(x-1)^k(1+Q_1),
\end{equation}
\end{example}

\begin{example}
The subfield of squares, $F_0(n)^2=\set{f\in F_0(n)}{f=1+\sum_{\nu\geq 1}\eta_\nu(x-1)^{2\nu}}$
is isomorphic to all subfields generated by a polynomial of the form $f_0=1+(x-1)^2 f_1$,
\end{example}

We will need

\begin{lemma}\label{binomial formula}
Fix a natural number $\alpha$, write $\alpha=\sum_{\nu=0}^{n_\alpha}\alpha_\nu 2^\nu$,
$\alpha_\nu=0,1$, and let
\begin{equation}
N_\alpha=\set{\sum_{\nu=0}^{n_\alpha}\eps_\nu\alpha_\nu 2^\nu}{\eps_\nu=0,1}.
\end{equation}
Then
\begin{equation}
\left[1+(x-1)^k\right]^\alpha=1+\sum_{n\in N_\alpha}(x-1)^{kn}.
\end{equation}
Conversely, given $N\sub\{1,\ldots,n-1\}$, then $P_N=1+\sum_{\nu\in N}(x-1)^{\nu}=(1+(x-1)^k)^\alpha$
iff $N\sub k\mathbb{N}$ and
\begin{equation}
N=\set{\sum_\nu\eps_\nu2^\nu}{k2^\nu\in N,\ \eps_\nu=0,1}\cap\{1,\ldots,n-1\}.
\end{equation}
\end{lemma}
\begin{proof}
We have
\begin{equation}
(1+(x-1)^k)^\alpha=\prod_{\nu=0}^{n_\alpha}\left(1+(x-1)^{k2^\nu}\right)^{\alpha_\nu},
\end{equation}
and each subset $N'\sub\set{\nu}{\alpha_\nu=1}$ corresponds to exactly on of the summands $(x-1)^{kn}$ of
$1+\sum_{n\in N_\alpha}(x-1)^{kn}$, where $n=\sum_{\nu\in N'}2^\nu$ (and with $n=0$ for the empty set).
\end{proof}

We start with a `local' version Theorem~\ref{singly-gen-char}.

\begin{lemma}\label{1-generated subfield}
The unital 3-subfield generated by $P=1+(x-1)^k(1+P_k)\in F_0(n)$ is given by
\begin{equation}
\langle P\rangle=
\set{1+\sum_{0<ki\leq n-1}\eta_i(x-1)^{ki}(1+P_k)^i}{\eta_i\in\mathbb{F}_2}=: F_1
\end{equation}
and is isomorphic to
\begin{equation}
\langle 1+(x-1)^k\rangle=\set{1+\sum_{0<i\leq(n-1)/k}\eta_i(x-1)^{ki}}{\eta_i\in\mathbb{F}_2}.
\end{equation}
\end{lemma}
\begin{proof}
By (a slight modification of) Lemma~\ref{binomial formula} the elements of $\langle P\rangle$ belong to $F_1$,
and since $F_1$ is a unital 3-field, $\langle P\rangle\sub F_1$.
By Theorem~\ref{singly-gen-char}, the isomorphism class of $\langle P\rangle$ is determined by the minimal number
$n_P$ with $(P-1)^{n_P}=0$. This number is the same for $P$ and $1+(x-1)^k$ hence
$\langle P\rangle\cong\langle 1+(x-1)^k\rangle$ are both of cardinality $2^{n_P-1}$.
Since also $F_1$ is of this cardinality, $F_1=\langle P\rangle$.
\end{proof}

\begin{theorem}
$F\sub F_0(n)$ is a subfield iff there are natural numbers $g_1,\ldots,g_k\leq n-1$, $k\leq n-1$, with
\begin{multline}
F\cong F_0(n;g_1,\ldots g_k):=\langle 1+(x-1)^{g_1},\ldots,1+(x-1)^{g_k}\rangle=\\
=\set{1+\sum_{1\leq \nu_1g_1+\ldots+\nu_kg_k\leq n-1}\eps_{\nu_1\nu_2\ldots \nu_k}(x-1)^{\nu_1g_1+\ldots+\nu_kg_k}}
             {\eps_{\nu_1\ldots \nu_k}\in\mathbb{F}_2}=:G(n;g_1,\ldots,g_k).
\end{multline}
\end{theorem}
\begin{proof}
In order to see that $F_0(n;g_1,\ldots,g_k)$ is unital 3-subfield it suffices to show that it is
a unital 3-subring. But this is obvious, and it follows that $F_0(n;g_1,\ldots,g_k)$ is contained
in $G(n;g_1,\ldots,g_k)$. By Lemma~\ref{1-generated subfield},
\begin{multline}
1+(x-1)^{\mu g_i+\nu g_j}=\\
=(1+(x-1)^{\mu g_i})(1+(x-1)^{\nu g_j})-(1+(x-1)^{\mu g_i})-(1+(x-1)^{\nu g_j})
\end{multline}
is contained in $F_0(n;g_1,\ldots g_k)$, and similarly, $1+(x-1)^{\nu_1g_1+\ldots+\nu_kg_k}\in F_0(n;g_1,\ldots,g_k)$.
Using induction over the number of elements in $\{\eps_{\nu_1,\ldots,\nu_k}=1\}$ with the induction step
of adjoining $1+(x-1)^{\langle\mu,g\rangle}$ to
$1+\sum_{1\leq\langle g,\nu\rangle\leq n-1}\eps_\nu(x-1)^{\langle\nu,g\rangle}$ being established by
\begin{multline}
1+(x-1)^{\langle\mu,g\rangle}+\sum_{1\leq\langle g,\nu\rangle\leq n-1}\eps_\nu(x-1)^{\langle\nu,g\rangle}=\\
=1+\left(1+(x-1)^{\langle\mu,g\rangle}\right)+
\left(1+\sum_{1\leq\langle g,\nu\rangle\leq n-1}\eps_\nu(x-1)^{\langle\nu,g\rangle}\right),
\end{multline}
one shows that, conversely, $G(n;g_1,\ldots,g_k)\sub F_0(n;g_1,\ldots,g_k)$.
\end{proof}

Let $F\sub F_0(n)$ be a subfield. We call the semigroup
\begin{equation}
\operatorname{Ex}F=\set{\alpha\in\mathbb{Z}/n}{1+(x-a)^\alpha\in F}
\end{equation}
the \emph{exponents} of $F$. The proof of the following results rest on

\begin{lemma}
Suppose $S$ is a sub-semigroup of the additive semigroup $\mathbb{Z}/n$ and
$N_0$ one of its subsets. Then
\begin{enumerate}
  \item there are uniquely determined elements $s_1<\ldots<s_k\in S$ with
  \begin{equation}
  s_{\varkappa+1}=\min S\setminus\set{z_1s_1+\ldots+z_\varkappa s_\varkappa}{z_i\in\mathbb{Z}/n}
  \end{equation}
  \begin{equation}
  S=\set{z_1s_1+\ldots+z_ks_k}{z_i\in\mathbb{Z}/n}
  \end{equation}
  Equivalently, none of the $s_i$ divides $s_j$.
  \item The sub-semigroup $S(N_0)$ generated by $N_0$ can be constructed inductively by letting $s_1=\min N_0$
  and $s_{n+1}=\min N_0\setminus\set{z_1s_1+\ldots z_ns_n}{z_1,\ldots z_n\in\mathbb{N}}$ as long as the latter
  set is not empty. If $M$ is the index before this happens,
\begin{equation}
  S(N_0)=\setminus\set{z_1s_1+\ldots z_Ms_M}{z_1,\ldots z_M\in\mathbb{N}}
\end{equation}
\end{enumerate}
\end{lemma}

\begin{corollary}
The set of isomorphism classes of subfields of $F_0(n)$ is order isomorphic to the set of sub-semigroups of the
additive semigroup $\mathbb{Z}/n$. Furthermore, two subfields $F_{1,2}$ of $F_0(n)$ are
isomorphic iff $\operatorname{Ex}F_1=\operatorname{Ex}F_2$.
\end{corollary}

\subsection{The group of automorphisms and the global involution}

For the following it is important to observe that the elements of
\begin{equation}
\mathcal{Q}F_0(n)=\set{\sum_{i=1}^{n-1}\eps_{i}\left(x-1\right)^{i}}{\eps_{i}\in\mathbb{F}_2}
\end{equation}
are in 1-1 correspondence with unital endomorphisms $\Phi$
of $F_0(n)$: Each such $\Phi$ is uniquely determined
by the polynomial $P_\Phi=\mathcal{Q}\Phi(x-1)\in\mathcal{Q}F_0(n)$ and, conversely,
any polynomial $P=1+P_0\in F_0(n)=1+\mathcal{Q}F_0(n)$ provides a unital endomorphisms $\Phi_P$,
defined for $Q=1+Q_0(x-1)\in F_0(n)$ by
\begin{equation}
\Phi_P(1+Q_0(x-1))=1+Q_0\circ P_0\mod (x-1)^n.
\end{equation}
This map is well-defined
because the ideal generated by $(x-1)^n$ for the ternary polynomial ring
$F_0[x]$ within
\begin{equation}
\mathcal{Q}F_0[x]=\set{\sum_{i=1}^{N}\eps_{i}\left(x-1\right)^{i}}{N\in\mathbb{N}, \eps_{i}\in\mathbb{F}_2}
\end{equation}
is invariant under composition with polynomials in $(x-1)$ from the right.
It is additive and respects multiplication because
\begin{multline}
\Phi_P\left(\left(1+Q_0(x-1)\right)\left(1+R_0(x-1)\right)\right)=\\
=1+\left(Q_0(x-1)+R_0(x-1)+Q_0(x-1)R_0(x-1)\right)\circ P_0=\\
=\Phi_P\left(1+Q_0(x-1)\right)\Phi_P\left(1+R_0(x-1)\right)
\end{multline}
Since $\Phi_Q\circ\Phi_P=\Phi_{Q\circ P}$ we have shown the first part of
\begin{theorem}
The mapping $\Phi\longmapsto\mathcal{Q}\Phi(x-1)$ establishes an isomorphism
between the ring of unital endomorphisms of $F_0(n)$ with respect to composition of maps
and $\mathcal{Q}F_0(n)$, equipped with composition of polynomials in the variable
$(x-1)$.

Under this identification, a polynomial $P$ corresponds to an automorphism of $F_0(n)$
iff there is a polynomial $Q\in F_0(n)$ such that
$P_0\circ Q_0=P_0\circ Q_0=(x-1)$ which is equivalent to
\begin{equation}
\mathcal{Q}\Phi(x-1)=(x-1)+\sum_{i=2}^{n-1}\eps_k(x-1)^k.
\end{equation}
\end{theorem}
\begin{proof}
We still have to show the final statement.
For arbitrary $n=\alpha_0+\alpha_1 2+\ldots+\alpha_N2^N\in\mathbb{N}$, $\alpha_\nu\in\mathbb{F}_2$
and any polynomial $P_0=\sum_{i=1}^{n-1}\eps_k(x-1)^k$
\begin{equation}
P_0^n=\left(\sum_{i=1}^{n-1}\eps_k(x-1)^k\right)^{\alpha_0}\left(\sum_{i=1}^{n-1}\eps_k(x-1)^{2k}\right)^{\alpha_1}
\ldots\left(\sum_{i=1}^{n-1}\eps_k(x-1)^{2^N}\right)^{\alpha_N},
\end{equation}
with some of the factors of highest order potentially equal to zero.
Accordingly, for no polynomial $Q_0$, $Q_0\circ P_0$ has $(x-1)$ as lowest term if $P_0$ doesn't,
and any polynomial $P_0$ giving rise to an automorphism must have $(x-1)$ as lowest term.
\footnote{Note that this is different for polynomials $P$ of the form $1+\sum_{k\geq 1}\eps_k x^k$,
since large powers of the polynomial $x$ might contain a constant term. For example, in $F_0(5)$,
$x^5=x^4+x-1$, and also $(x+x^2+x^4)\circ (1+x+x^2)=x$.} (This could also be shown by using the ideal
structure of $F_0(n)$.)

For the converse, we represent $\Phi_P$ with $P_0=(x-1)^k(1+P_1)$ as a matrix with respect to
the basis $(x-1)^\nu$, $\nu=1,\ldots,n-1$, of $\mathcal{U}(F_0(n))$. Writing
\begin{equation}
\Phi_P\left[(x-1)^\nu\right]=(x-1)^\nu(1+P_1)^\nu
\end{equation}
as row vectors, it turns out that the resulting matrix is upper triangular with 1's on the main diagonal so
that $\Phi_P$ must be invertible.
\end{proof}

Using the multinomial identity
\begin{equation}
(a_1+\ldots a_m)^n=\sum_{i_1+\ldots +i_m=n}\frac{n!}{i_1!\ldots i_m!}a_1^{i_1}\ldots a_m^{i_m}
\end{equation}
we find for the polynomial $P_0=(x-1)+\sum_{\nu=2}^{n-1}\eps_\nu(x-1)^\nu$
\begin{gather}
P_0^k=\sum_{i_1+\ldots i_{n-1}=k}\frac{k!}{i_1!\ldots i_{n-1}!}\eps_2\ldots\eps_{n-1}
                                                           (x-1)^{i_1+2i_2+\ldots+(n-1)i_{n-1}}=\\
=\sum_{\ell=k}^{n-1}\left(\sum_{i_1+2i_2+\ldots +(n-1)i_{n-1}=\ell}
                         \frac{k!}{i_1!\ldots i_{n-1}!}\eps_2\ldots\eps_{n-1}\right)(x-1)^\ell
\end{gather}
On the other hand, over $\mathbb{F}_2$, and with $n=\sum_{\nu=0}^N n_\nu 2^\nu$,
\begin{equation}
(a_1+\ldots a_m)^n=\prod_{\nu=0}^N\left(a_1^{2^\nu}+\ldots+a_m^{2^\nu}\right)^{n_\nu}
\end{equation}
and
\begin{equation}
P_0^k=\prod_{\nu=0}^N\left((x-1)^{2^\nu}+\ldots+\eps_\mu(x-1)^{2^\nu\mu}\ldots+\eps_m(x-1)^{2^\nu m}\right)^{n_\nu}
\end{equation}

\begin{corollary}
The coefficients of $A_P=(\alpha_{k\ell}^P)$ are
\begin{equation}
  \alpha_{k\ell}^P=
  \sum_{i_1+2i_2+\ldots +(n-1)i_{n-1}=\ell}\frac{k!}{i_1!\ldots i_{n-1}!}\eps_2\ldots\eps_{n-1},
\end{equation}
for $k=1,\ldots,n-1$, $\ell=k,\ldots,n-1$.
\end{corollary}

If $o(\Phi_P)=2^{\ell(\Phi_P)}$ denotes the order of $\Phi_P\in\Aut F_0(n)$ then
$\Phi_P^{-1}=\Phi_P^{o(\Phi_P)-1}$. An algorithm of lower complexity is obtained using matrix products
\begin{corollary}
Let $P=1+(x-1)+P_1=\sum_{\nu=0}^n\eps_\nu (x-1)^\nu$ be the polynomial that belongs to $\Phi_P\in\Aut(F)$. Then
the coefficients of the polynomial $\widehat{P}=\sum_{\nu=0}^n\widehat{\eps}_\nu (x-1)^\nu$ for which
$\Phi_P^{-1}=\Phi_{\widehat{P}}$ can be recursively calculated using
\begin{equation}
\widehat{\eps}_0=\widehat{\eps}_1=1\qquad
\widehat{\eps}_{\ell}=\widehat{\eps}_1\eps^{(1)}_{\ell}+\widehat{\eps}_{2}\eps^{(2)}_{\ell}
+\ldots+\widehat{\eps}_{\ell-1}\eps^{(\ell-1)}_{\ell}
\end{equation}
\end{corollary}

\begin{lemma}\label{universal reflection}
Let $s_n=-\sum_{\nu=1}^{n-1}(x-1)^\nu\in \mathcal{Q}F_0(n)$. Then
$R_n:=\Phi_{s_n}$ has order 2 and
\begin{equation}
R_n\left[1+\sum_{i=1}^{n-1}\eps_i(x-1)^i\right]=1+\sum_{i=1}^{n-1}\eps_i\left((x-1)^\hash\right)^i.
\end{equation}
\end{lemma}
\begin{proof}
Since, with respect to multiplication,
\begin{equation}
s_n=1+\left[(x-1)-1\right]^{-1}
\end{equation}
it follows that $1+(s_n-1)^{-1}=(x-1)$ or,
$s_n\circ s_n=(x-1)$ (which is equivalent to observing that $q\mapsto q^\hash$ is an involution,
Definition~\ref{hash-def} and Example~\ref{hash-ex}).
The somewhat more explicit form for $R_n$ is a consequence of the fact that $s_n=(x-1)^\hash$.
\end{proof}

\begin{lemma}\label{semi-direct-char}
A group $G$ of order $2n$ has a representation $\mathbb{Z}/2\rtimes H$ iff $G$ contains
a normal subgroup $H$ of order $n$ as well as an involution $r\in G\setminus H$.
\end{lemma}

\begin{proof}
The short exact sequence $H\rightarrow G\rightarrow\mathbb{Z}/2$ can be split by mapping
the non-unit of $\mathbb{Z}/2$ to $r$.
\end{proof}

\begin{lemma}\label{Gamma}
Define
\begin{equation}
M_{n,k}:\Aut F_0(n)\to\Aut F_0(k),\qquad
M_{n,k}\Phi(x+\mathfrak{I}_k)=\Phi(x)+\mathfrak{I}_k,
\end{equation}
where $1+\sum_{\nu=1}^{n-1}\eps_\nu(x-1)^\nu +\mathfrak{I}_k\in F_0(n)/\mathfrak{I}_k$ is identified with
$1+\sum_{\nu=1}^{k-1}\eps_\nu(x-1)^\nu \in F_0(k)$.
\begin{enumerate}
\item Identifying $F_0(k)$ with $F_0(n)/\mathfrak{I}_k$, if $\Phi=\Phi_P$ with
  $P=(x-1)+\sum_{\nu=2}^{n-1}\eps_\nu(x-1)^\nu$ then
  \begin{equation}
  M_{n,k}\Phi_P=(x-1)+\sum_{\nu=2}^{k-1}\eps_\nu(x-1)^\nu,
  \end{equation}
  and $\Ker M_{n,k}$ is equal to
  \begin{equation}
  \Gamma_{n,k}=\set{\Phi_P\in\Aut F_0(n)}{P_0=(x-1)+\sum_{\nu=k}^{n-1}\eps_\nu(x-1)^\nu,\ \eps_\nu\in\mathbb{F}_2}
  \end{equation}
\item Restricting $M_{n,k+1}$ to $\Gamma_{n,k}$ yields an exact sequence
  \begin{equation}
  0\longrightarrow\Gamma_{n,k+1}\longrightarrow\Gamma_{n,k}\longrightarrow\Gamma_{k+1,k}\longrightarrow 0
  \end{equation}
  This sequence splits whenever $k>2$ so that in this case,
  \begin{equation}
  \Gamma_{n,k}\cong\Gamma_{n,k+1}\rtimes\mathbb{F}_2.
  \end{equation}
\item Similarly, $\varphi_n$ generates the group morphism
  \begin{equation}
  \varrho_n:\Aut F_0(n)\to\Aut\left(F_0(n)^2\right),\qquad \varrho_n(\Phi)=\rest{\Phi}{F_0(n)^2}
  \end{equation}
  Identifying the subfield of squares in $F_0(n)$ with $F_0\left(\left\lfloor\frac{n}{2}\right\rfloor\right)$,
  the effect of $\varrho_n$ on the polynomial $P$ representing $\Phi=\Phi_P$ can be seen as
  \begin{equation}
  \varrho_n:\Aut F_0(n)\to\Aut F_0\left(\left\lfloor\frac{n}{2}\right\rfloor\right),\qquad
  P\mapsto\rest{P}{F_0(\lfloor\frac{n}{2}\rfloor)}
  \end{equation}
  In this picture, $\Phi_P\in\Ker\varrho_n$ iff $P=(x-1)+\sum_{\nu\geq\lceil n/2\rceil}\eps_\nu(x-1)^\nu$.
  \item
  Accordingly, for each $k$ with $2^k\leq n$ there is a short exact sequence
  \begin{equation}
  G_{n,k}\longrightarrow\Aut F_0(n)\longrightarrow\Aut\left((F_0(n)^{2^k}\right)
  \cong
  \Aut F_0\left(\left\lfloor\frac{n}{2^k}\right\rfloor\right)
  \end{equation}
  where $G_{n,k}$ is the normal subgroup of automorphisms for which
  $1+(x-1)^{2^k}$ is a fixed point.
  The polynomials $P_0$, representing $\Phi\in G_{n,k}$ via $\Phi(1+(x-1))=1+P_0(x-1)$, are characterized by
\begin{multline}
  P_0=(x-1)+\sum_{i>(n-1)/2^k}\eps_i(x-1)^i=\\
  (x-1)+(x-1)^{\lceil(n-1)/2^k\rceil}
  \sum_{i=1}^{n-1-\lceil(n-1)/2^k\rceil}\eps_i(x-1)^i,
  \qquad \eps_i\in\mathbb{F}_2,
\end{multline}
  and it follows that $G_{n,1}$ is commutative.
\end{enumerate}
\end{lemma}

\begin{proof}
\textbf{(1)}
Because $\mathfrak{I}_k=\Ker\mu_{n,k}=\set{\sum_{i=k+1}^{n-1}\eps_i(x-1)^i}{\eps_i\in\mathbb{F}_2}$ is invariant under substitution by any
polynomial $Q_0=(x-1)+\sum_{i=2}^n\eta_i(x-1)^i$ (which also follows from Lemma~\ref{ideals in F_0(n)}, each $\Phi\in\Aut F_0(n)$ acts on
$F_0(n)/\Ker\mu_{n,k}=F_0(k)$ by substitution and an application of $\mu_{n,k}$. The resulting element
$M_{n,k}(\Phi)\in\Aut F_0(k)$ then maps $x-1$ to $\mu_{n,k}P_0$ for $\Phi=\Phi_P\in\Aut F_0(n)$.
Consequently, $M_{n,k}:\Aut F_0(n)\to\Aut F_0(k)$ is a morphism with
kernel consisting of automorphisms $\Phi_P\in\Aut F_0(n)$ such that for all $Q=1+\sum_{i=1}^{k-1} q_i(x-1)^i$
\begin{equation}
\mu_{n,k}\Phi_P(Q)=\mu_{n,k}\left(1+Q_0\circ P_0\right)=Q.
\end{equation}
which shows that $P=1+(x-1)+\sum_{i=k}^n\eta_i(x-1)^i$.

\textbf{(2)}
The short exact sequence is established using (1).
Polynomials for elements of $\Gamma_{n,k}\setminus\Gamma_{n,k+1}$ are of the form
$P=(x-1)+(x-1)^k+\sum_{i=k+1}^{n-1}\eps_i(x-1)^i$ so that $R_n\not\in\Gamma_{n,k}\setminus\Gamma_{n,k+l}$
if $k>2$. The claim follows from Lemma~\ref{universal reflection} and Lemma~\ref{semi-direct-char}.

\textbf{(3)}
$\varrho_n(\Phi)$ is a well-defined morphism. Identifying $F_0(\lfloor n/2\rfloor)$ with $F_0(n)^2$
via $\sigma_n(P)=P^2$, one has for $P=(x-1)+\sum_{\nu=2}^{n-1}\eps_\nu(x-1)^\nu$
\begin{equation}
\sigma_n^{-1}\varrho_n(\Phi_P)\sigma_n(x-1)=\sigma_n^{-1}P^2=
(x-1)+\sum_{\nu=2}^{\lfloor n/2\rfloor}\eps_\nu(x-1)^\nu
\end{equation}
The elements $\Phi_P$ in the kernel $G_{n,k}$
of $\varrho^k_n$ restrict to all
polynomials $\sum_{i2^k\leq n-1}\eps_i(x-1)^{i2^k}$ as the identity or,
$\sum_{i2^k\leq n-1}\eps_iP_0(x-1)^{i2^k}=\sum_{i2^k\leq n-1}\eps_i(x-1)^{i2^k}$.
Equivalently, $P_0^{2^k}=(x-1)^{2^k}$, showing that $P_0=(x-1)+\sum_{n\leq i2^k}\eps_i(x-1)^i$.
The commutativity of $G_{n,1}$ is a consequence of the fact that for $P_0,Q_0\in G_{n,1}$
we have
\begin{equation}
P_0\circ Q_0(x-1)=Q_0(x-1)+\sum_{i=\lceil n/2\rceil}^{n-1}\eps_i Q_0^i(x-1)=Q_0(x-1)+P_0(x-1)
\end{equation}
\end{proof}

\begin{corollary}
For each $n\geq 3$,
\begin{equation}
\Aut F_0(n)\cong\mathbb{F}_2\rtimes\mathbb{F}_2\ldots\rtimes\mathbb{F}_2,
\end{equation}
an $(n-2)$-fold semi-direct product of $\mathbb{F}_2$.
\end{corollary}

\begin{proof}
By induction, this is a consequence of Lemma~\ref{morphisms}~(2).
\end{proof}

This last result by itself does not reveal much of the structure of $\Aut F_0(n)$:
Analogous results hold for each nilpotent Lie groups and, in all likelihood, it is possible
to obtain the former by showing that $\Aut F_0(n)$ is a (nilpotent) Lie group in characteristic 2.
More information can be gained by further investigating the involution that is underlying the proof
of the above result.

\begin{definition}\label{Conjugation}
Let $f\in F_0(n)$ as well as $\Phi\in\Aut F_0(n)$.
\begin{enumerate}
  \item Define conjugations $f^\ast$ and $\Phi^\ast$ by
  \begin{equation}
    f^\ast=R_n(f)\qquad \Phi^\ast(f)=\Phi(f^\ast)^\ast,
  \end{equation}
  \item and we denote the 1-eigenspaces of these involutions by
\begin{equation}
  F_0(n)_1=\set{f\in F_0(n)}{f^\ast=f},\qquad \Aut F_0(n)_1=\set{\Phi\in\Aut F_0(n)}{\Phi^*=\Phi}
\end{equation}
\end{enumerate}
\end{definition}

\begin{example}
All elements of $\mathfrak{J}_k$ with $k\geq n/2$ are contained in $F_0(n)_1$, because if $f=(x-1)^k$
then, for $k=\sum_{n_0}\alpha_\nu 2^\nu$
\begin{equation}
f^\ast=s_n^k=\prod_{n_0}\left(\sum_{\mu=1}^{n-1}(x-1)^{\nu\alpha_\nu 2^\nu}\right)
\end{equation}
\end{example}

\appendix

\section{A representation of the truncated Nottingham groups $\Aut F_0(3)$ -- $\Aut F_0(7)$}
The case $n=2$ is trivial, for $n=3$ the polynomials $1+(x-1)$ and $P(x)=1+(x-1)+(x-1)^2$ give rise to
identity as well as to the automorphism with representing matrix
\begin{equation}
A_P=\begin{pmatrix}1 & 1 \\ 0 & 1\end{pmatrix}
\end{equation}
If $n=4$, besides the identity, $P_1=1+(x-1)+(x-1)^2$, $P_2=1+(x-1)+(x-1)^3$ and $P_3=1+(x-1)+(x-1)^2+(x-1)^3$
correspond to
\begin{equation}
A_1=\begin{pmatrix}1 & 1 & 0\\ 0 & 1 & 0\\ 0 & 0 & 1\end{pmatrix}\quad
A_2=\begin{pmatrix}1 & 0 & 1\\ 0 & 1 & 0\\ 0 & 0 & 1\end{pmatrix}\quad
A_3=\begin{pmatrix}1 & 1 & 1\\ 0 & 1 & 0\\ 0 & 0 & 1\end{pmatrix},
\end{equation}
which is the cyclic group $C_4$.
For $n=5$, $P_\alpha=1+(x-1)+\sum_{i=2}^4\alpha_{i-2}(x-1)^i$, with $(\alpha_0,\alpha_1,\alpha_2)$
representing the binary representation of $\alpha$, produces the matrices
\begin{gather}
A_0=\begin{pmatrix}1 & 0 & 0 & 0\\ 0 & 1 & 0 & 0\\ 0 & 0 & 1 & 0 \\ 0 & 0 & 0 & 1\end{pmatrix}\quad
A_1=\begin{pmatrix}1 & 1 & 0 & 0\\ 0 & 1 & 0 & 1\\ 0 & 0 & 1 & 1 \\ 0 & 0 & 0 & 1\end{pmatrix}\quad
A_2=\begin{pmatrix}1 & 0 & 1 & 0\\ 0 & 1 & 0 & 0\\ 0 & 0 & 1 & 0 \\ 0 & 0 & 0 & 1\end{pmatrix}\quad
A_3=\begin{pmatrix}1 & 1 & 1 & 0\\ 0 & 1 & 0 & 1\\ 0 & 0 & 1 & 1 \\ 0 & 0 & 0 & 1\end{pmatrix}\\
A_4=\begin{pmatrix}1 & 0 & 0 & 1\\ 0 & 1 & 0 & 0\\ 0 & 0 & 1 & 0 \\ 0 & 0 & 0 & 1\end{pmatrix}\quad
A_5=\begin{pmatrix}1 & 1 & 0 & 1\\ 0 & 1 & 0 & 1\\ 0 & 0 & 1 & 1 \\ 0 & 0 & 0 & 1\end{pmatrix}\quad
A_6=\begin{pmatrix}1 & 0 & 1 & 1\\ 0 & 1 & 0 & 0\\ 0 & 0 & 1 & 0 \\ 0 & 0 & 0 & 1\end{pmatrix}\quad
A_7=\begin{pmatrix}1 & 1 & 1 & 1\\ 0 & 1 & 0 & 1\\ 0 & 0 & 1 & 1 \\ 0 & 0 & 0 & 1\end{pmatrix}.
\end{gather}
As observed earlier, this group is isomorphic to the dihedral group of order 8, i.e. $C_4\rtimes C_2$.
Similarly for $n=6$, $P_\alpha=1+(x-1)+\sum_{i=2}^5\alpha_{i-2}(x-1)^i$, with $(\alpha_0,\alpha_1,\alpha_2,
\alpha_3)$
again the binary representation of $\alpha$ we have
{\tiny
\begin{align}
A_0  &  =\begin{pmatrix}1 & 0 & 0 & 0 & 0\\ 0 & 1 & 0 & 0 & 0\\ 0 & 0 & 1 & 0 & 0\\ 0 & 0 & 0 & 1 & 0\\ 0 & 0 & 0 & 0 & 1\end{pmatrix}\quad
A_1     =\begin{pmatrix}1 & 1 & 0 & 0 & 0\\ 0 & 1 & 0 & 1 & 0\\ 0 & 0 & 1 & 1 & 1\\ 0 & 0 & 0 & 1 & 0\\ 0 & 0 & 0 & 0 & 1\end{pmatrix}\quad
A_2       =\begin{pmatrix}1 & 0 & 1 & 0 & 0\\ 0 & 1 & 0 & 0 & 0\\ 0 & 0 & 1 & 0 & 1\\ 0 & 0 & 0 & 1 & 0\\ 0 & 0 & 0 & 0 & 1\end{pmatrix}\quad
A_3       =\begin{pmatrix}1 & 1 & 1 & 0 & 0\\ 0 & 1 & 0 & 1 & 0\\ 0 & 0 & 1 & 1 & 0\\ 0 & 0 & 0 & 1 & 0\\ 0 & 0 & 0 & 0 & 1\end{pmatrix}\\
A_4  &   =\begin{pmatrix}1 & 0 & 0 & 1 & 0\\ 0 & 1 & 0 & 0 & 0\\ 0 & 0 & 1 & 0 & 0\\  0 & 0 & 0 & 1 & 0\\ 0 & 0 & 0 & 0 & 1\end{pmatrix}\quad
A_5    =\begin{pmatrix}1 & 1 & 0 & 1 & 0\\ 0 & 1 & 0 & 1 & 0\\ 0 & 0 & 1 & 1 & 1\\  0 & 0 & 0 & 1 & 0\\ 0 & 0 & 0 & 0 & 1\end{pmatrix}\quad
A_6      =\begin{pmatrix}1 & 0 & 1 & 1 & 0\\ 0 & 1 & 0 & 0 & 0\\ 0 & 0 & 1 & 0 & 1\\ 0 & 0 & 0 & 1 & 0\\ 0 & 0 & 0 & 0 & 1\end{pmatrix}\quad
A_7      =\begin{pmatrix}1 & 1 & 1 & 1 & 0\\ 0 & 1 & 0 & 1 & 0\\ 0 & 0 & 1 & 1 & 0\\ 0 & 0 & 0 & 1 & 0\\ 0 & 0 & 0 & 0 & 1\end{pmatrix}\\
A_8  &   =\begin{pmatrix}1 & 0 & 0 & 0 & 1\\ 0 & 1 & 0 & 0 & 0\\ 0 & 0 & 1 & 0 & 0\\ 0 & 0 & 0 & 1 & 0\\ 0 & 0 & 0 & 0 & 1\end{pmatrix}\quad
A_9     =\begin{pmatrix}1 & 1 & 0 & 0 & 1\\ 0 & 1 & 0 & 1 & 0\\ 0 & 0 & 1 & 1 & 1\\ 0 & 0 & 0 & 1 & 0\\ 0 & 0 & 0 & 0 & 1\end{pmatrix}\quad
A_{10}  =\begin{pmatrix}1 & 0 & 1 & 0 & 1\\ 0 & 1 & 0 & 0 & 0\\ 0 & 0 & 1 & 0 & 1\\ 0 & 0 & 0 & 1 & 0\\ 0 & 0 & 0 & 0 & 1\end{pmatrix}\quad
A_{11}  =\begin{pmatrix}1 & 1 & 1 & 0 & 1\\ 0 & 1 & 0 & 1 & 0\\ 0 & 0 & 1 & 1 & 0\\ 0 & 0 & 0 & 1 & 0\\ 0 & 0 & 0 & 0 & 1\end{pmatrix}\\
A_{12}& =\begin{pmatrix}1 & 0 & 0 & 1 & 1\\ 0 & 1 & 0 & 0 & 0\\ 0 & 0 & 1 & 0 & 0\\ 0 & 0 & 0 & 1 & 0\\ 0 & 0 & 0 & 0 & 1\end{pmatrix}\quad
A_{13} =\begin{pmatrix}1 & 1 & 0 & 1 & 1\\ 0 & 1 & 0 & 1 & 0\\ 0 & 0 & 1 & 1 & 1\\ 0 & 0 & 0 & 1 & 0\\ 0 & 0 & 0 & 0 & 1\end{pmatrix}\quad
A_{14}  =\begin{pmatrix}1 & 0 & 1 & 1 & 1\\ 0 & 1 & 0 & 0 & 0\\ 0 & 0 & 1 & 0 & 1\\ 0 & 0 & 0 & 1 & 0\\ 0 & 0 & 0 & 0 & 1\end{pmatrix}\quad
A_{15}  =\begin{pmatrix}1 & 1 & 1 & 1 & 1\\ 0 & 1 & 0 & 1 & 0\\ 0 & 0 & 1 & 1 & 0\\ 0 & 0 & 0 & 1 & 0\\ 0 & 0 & 0 & 0 & 1\end{pmatrix}
\end{align}
}
Calculating the cycle graph of this group, it turns out to be $G_{16}\!^3=K_4\rtimes C_4=(C_4\times C_2)\rtimes C_2$,
where $K_4$ denotes the Klein 4-group,
\href{https://en.wikipedia.org/wiki/List_of_small_groups#List_of_small_non-abelian_groups}{(Wikipedia, List of small groups)}.

\newpage

\vspace*{\fill}

We follow the same procedure in case $n=7$ and obtain

\vspace*{\fill}

{\tiny
\begin{align}
A_{0}&=
\begin{pmatrix}
 1 & 0 & 0 & 0 & 0 & 0 \\
 0 & 1 & 0 & 0 & 0 & 0 \\
 0 & 0 & 1 & 0 & 0 & 0 \\
 0 & 0 & 0 & 1 & 0 & 0 \\
 0 & 0 & 0 & 0 & 1 & 0 \\
 0 & 0 & 0 & 0 & 0 & 1 \\
\end{pmatrix}
\quad
A_{1}=
\begin{pmatrix}
 1 & 1 & 0 & 0 & 0 & 0 \\
 0 & 1 & 0 & 1 & 0 & 0 \\
 0 & 0 & 1 & 1 & 1 & 1 \\
 0 & 0 & 0 & 1 & 0 & 0 \\
 0 & 0 & 0 & 0 & 1 & 1 \\
 0 & 0 & 0 & 0 & 0 & 1 \\
\end{pmatrix}
\quad
A_{2}=
\begin{pmatrix}
 1 & 0 & 1 & 0 & 0 & 0 \\
 0 & 1 & 0 & 0 & 0 & 1 \\
 0 & 0 & 1 & 0 & 1 & 0 \\
 0 & 0 & 0 & 1 & 0 & 0 \\
 0 & 0 & 0 & 0 & 1 & 0 \\
 0 & 0 & 0 & 0 & 0 & 1 \\
\end{pmatrix}
\quad
A_{3}=
\begin{pmatrix}
 1 & 1 & 1 & 0 & 0 & 0 \\
 0 & 1 & 0 & 1 & 0 & 1 \\
 0 & 0 & 1 & 1 & 0 & 1 \\
 0 & 0 & 0 & 1 & 0 & 0 \\
 0 & 0 & 0 & 0 & 1 & 1 \\
 0 & 0 & 0 & 0 & 0 & 1 \\
\end{pmatrix}
\\
A_{4}&=
\begin{pmatrix}
 1 & 0 & 0 & 1 & 0 & 0 \\
 0 & 1 & 0 & 0 & 0 & 0 \\
 0 & 0 & 1 & 0 & 0 & 1 \\
 0 & 0 & 0 & 1 & 0 & 0 \\
 0 & 0 & 0 & 0 & 1 & 0 \\
 0 & 0 & 0 & 0 & 0 & 1 \\
\end{pmatrix}
\quad
A_{5}=
\begin{pmatrix}
 1 & 1 & 0 & 1 & 0 & 0 \\
 0 & 1 & 0 & 1 & 0 & 0 \\
 0 & 0 & 1 & 1 & 1 & 0 \\
 0 & 0 & 0 & 1 & 0 & 0 \\
 0 & 0 & 0 & 0 & 1 & 1 \\
 0 & 0 & 0 & 0 & 0 & 1 \\
\end{pmatrix}
\quad
A_{6}=
\begin{pmatrix}
 1 & 0 & 1 & 1 & 0 & 0 \\
 0 & 1 & 0 & 0 & 0 & 1 \\
 0 & 0 & 1 & 0 & 1 & 1 \\
 0 & 0 & 0 & 1 & 0 & 0 \\
 0 & 0 & 0 & 0 & 1 & 0 \\
 0 & 0 & 0 & 0 & 0 & 1 \\
\end{pmatrix}
\quad
A_{7}=
\begin{pmatrix}
 1 & 1 & 1 & 1 & 0 & 0 \\
 0 & 1 & 0 & 1 & 0 & 1 \\
 0 & 0 & 1 & 1 & 0 & 0 \\
 0 & 0 & 0 & 1 & 0 & 0 \\
 0 & 0 & 0 & 0 & 1 & 1 \\
 0 & 0 & 0 & 0 & 0 & 1 \\
\end{pmatrix}
\\
A_{8}&=
\begin{pmatrix}
 1 & 0 & 0 & 0 & 1 & 0 \\
 0 & 1 & 0 & 0 & 0 & 0 \\
 0 & 0 & 1 & 0 & 0 & 0 \\
 0 & 0 & 0 & 1 & 0 & 0 \\
 0 & 0 & 0 & 0 & 1 & 0 \\
 0 & 0 & 0 & 0 & 0 & 1 \\
\end{pmatrix}
\quad
A_{9}=
\begin{pmatrix}
 1 & 1 & 0 & 0 & 1 & 0 \\
 0 & 1 & 0 & 1 & 0 & 0 \\
 0 & 0 & 1 & 1 & 1 & 1 \\
 0 & 0 & 0 & 1 & 0 & 0 \\
 0 & 0 & 0 & 0 & 1 & 1 \\
 0 & 0 & 0 & 0 & 0 & 1 \\
\end{pmatrix}
\quad
A_{10}=
\begin{pmatrix}
 1 & 0 & 1 & 0 & 1 & 0 \\
 0 & 1 & 0 & 0 & 0 & 1 \\
 0 & 0 & 1 & 0 & 1 & 0 \\
 0 & 0 & 0 & 1 & 0 & 0 \\
 0 & 0 & 0 & 0 & 1 & 0 \\
 0 & 0 & 0 & 0 & 0 & 1 \\
\end{pmatrix}
\quad
A_{11}=
\begin{pmatrix}
 1 & 1 & 1 & 0 & 1 & 0 \\
 0 & 1 & 0 & 1 & 0 & 1 \\
 0 & 0 & 1 & 1 & 0 & 1 \\
 0 & 0 & 0 & 1 & 0 & 0 \\
 0 & 0 & 0 & 0 & 1 & 1 \\
 0 & 0 & 0 & 0 & 0 & 1 \\
\end{pmatrix}
\\
A_{12}&=
\begin{pmatrix}
 1 & 0 & 0 & 1 & 1 & 0 \\
 0 & 1 & 0 & 0 & 0 & 0 \\
 0 & 0 & 1 & 0 & 0 & 1 \\
 0 & 0 & 0 & 1 & 0 & 0 \\
 0 & 0 & 0 & 0 & 1 & 0 \\
 0 & 0 & 0 & 0 & 0 & 1 \\
\end{pmatrix}
\quad
A_{13}=
\begin{pmatrix}
 1 & 1 & 0 & 1 & 1 & 0 \\
 0 & 1 & 0 & 1 & 0 & 0 \\
 0 & 0 & 1 & 1 & 1 & 0 \\
 0 & 0 & 0 & 1 & 0 & 0 \\
 0 & 0 & 0 & 0 & 1 & 1 \\
 0 & 0 & 0 & 0 & 0 & 1 \\
\end{pmatrix}
\quad
A_{14}=
\begin{pmatrix}
 1 & 0 & 1 & 1 & 1 & 0 \\
 0 & 1 & 0 & 0 & 0 & 1 \\
 0 & 0 & 1 & 0 & 1 & 1 \\
 0 & 0 & 0 & 1 & 0 & 0 \\
 0 & 0 & 0 & 0 & 1 & 0 \\
 0 & 0 & 0 & 0 & 0 & 1 \\
\end{pmatrix}
\quad
A_{15}=
\begin{pmatrix}
 1 & 1 & 1 & 1 & 1 & 0 \\
 0 & 1 & 0 & 1 & 0 & 1 \\
 0 & 0 & 1 & 1 & 0 & 0 \\
 0 & 0 & 0 & 1 & 0 & 0 \\
 0 & 0 & 0 & 0 & 1 & 1 \\
 0 & 0 & 0 & 0 & 0 & 1 \\
\end{pmatrix}
\\
A_{16}&=
\begin{pmatrix}
 1 & 0 & 0 & 0 & 0 & 1 \\
 0 & 1 & 0 & 0 & 0 & 0 \\
 0 & 0 & 1 & 0 & 0 & 0 \\
 0 & 0 & 0 & 1 & 0 & 0 \\
 0 & 0 & 0 & 0 & 1 & 0 \\
 0 & 0 & 0 & 0 & 0 & 1 \\
\end{pmatrix}
\quad
A_{17}=
\begin{pmatrix}
 1 & 1 & 0 & 0 & 0 & 1 \\
 0 & 1 & 0 & 1 & 0 & 0 \\
 0 & 0 & 1 & 1 & 1 & 1 \\
 0 & 0 & 0 & 1 & 0 & 0 \\
 0 & 0 & 0 & 0 & 1 & 1 \\
 0 & 0 & 0 & 0 & 0 & 1 \\
\end{pmatrix}
\quad
A_{18}=
\begin{pmatrix}
 1 & 0 & 1 & 0 & 0 & 1 \\
 0 & 1 & 0 & 0 & 0 & 1 \\
 0 & 0 & 1 & 0 & 1 & 0 \\
 0 & 0 & 0 & 1 & 0 & 0 \\
 0 & 0 & 0 & 0 & 1 & 0 \\
 0 & 0 & 0 & 0 & 0 & 1 \\
\end{pmatrix}
\quad
A_{19}=
\begin{pmatrix}
 1 & 0 & 0 & 1 & 0 & 1 \\
 0 & 1 & 0 & 0 & 0 & 0 \\
 0 & 0 & 1 & 0 & 0 & 1 \\
 0 & 0 & 0 & 1 & 0 & 0 \\
 0 & 0 & 0 & 0 & 1 & 0 \\
 0 & 0 & 0 & 0 & 0 & 1 \\
\end{pmatrix}
\\
A_{20}&=
\begin{pmatrix}
 1 & 0 & 0 & 0 & 1 & 1 \\
 0 & 1 & 0 & 0 & 0 & 0 \\
 0 & 0 & 1 & 0 & 0 & 0 \\
 0 & 0 & 0 & 1 & 0 & 0 \\
 0 & 0 & 0 & 0 & 1 & 0 \\
 0 & 0 & 0 & 0 & 0 & 1 \\
\end{pmatrix}
\quad
A_{21}=
\begin{pmatrix}
 1 & 1 & 1 & 0 & 0 & 1 \\
 0 & 1 & 0 & 1 & 0 & 1 \\
 0 & 0 & 1 & 1 & 0 & 1 \\
 0 & 0 & 0 & 1 & 0 & 0 \\
 0 & 0 & 0 & 0 & 1 & 1 \\
 0 & 0 & 0 & 0 & 0 & 1 \\
\end{pmatrix}
\quad
A_{22}=
\begin{pmatrix}
 1 & 1 & 0 & 1 & 0 & 1 \\
 0 & 1 & 0 & 1 & 0 & 0 \\
 0 & 0 & 1 & 1 & 1 & 0 \\
 0 & 0 & 0 & 1 & 0 & 0 \\
 0 & 0 & 0 & 0 & 1 & 1 \\
 0 & 0 & 0 & 0 & 0 & 1 \\
\end{pmatrix}
\quad
A_{23}=
\begin{pmatrix}
 1 & 1 & 0 & 0 & 1 & 1 \\
 0 & 1 & 0 & 1 & 0 & 0 \\
 0 & 0 & 1 & 1 & 1 & 1 \\
 0 & 0 & 0 & 1 & 0 & 0 \\
 0 & 0 & 0 & 0 & 1 & 1 \\
 0 & 0 & 0 & 0 & 0 & 1 \\
\end{pmatrix}
\\
A_{24}&=
\begin{pmatrix}
 1 & 0 & 1 & 1 & 0 & 1 \\
 0 & 1 & 0 & 0 & 0 & 1 \\
 0 & 0 & 1 & 0 & 1 & 1 \\
 0 & 0 & 0 & 1 & 0 & 0 \\
 0 & 0 & 0 & 0 & 1 & 0 \\
 0 & 0 & 0 & 0 & 0 & 1 \\
\end{pmatrix}
\quad
A_{25}=
\begin{pmatrix}
 1 & 0 & 1 & 0 & 1 & 1 \\
 0 & 1 & 0 & 0 & 0 & 1 \\
 0 & 0 & 1 & 0 & 1 & 0 \\
 0 & 0 & 0 & 1 & 0 & 0 \\
 0 & 0 & 0 & 0 & 1 & 0 \\
 0 & 0 & 0 & 0 & 0 & 1 \\
\end{pmatrix}
\quad
A_{26}=
\begin{pmatrix}
 1 & 0 & 0 & 1 & 1 & 1 \\
 0 & 1 & 0 & 0 & 0 & 0 \\
 0 & 0 & 1 & 0 & 0 & 1 \\
 0 & 0 & 0 & 1 & 0 & 0 \\
 0 & 0 & 0 & 0 & 1 & 0 \\
 0 & 0 & 0 & 0 & 0 & 1 \\
\end{pmatrix}
\quad
A_{27}=
\begin{pmatrix}
 1 & 1 & 1 & 1 & 0 & 1 \\
 0 & 1 & 0 & 1 & 0 & 1 \\
 0 & 0 & 1 & 1 & 0 & 0 \\
 0 & 0 & 0 & 1 & 0 & 0 \\
 0 & 0 & 0 & 0 & 1 & 1 \\
 0 & 0 & 0 & 0 & 0 & 1 \\
\end{pmatrix}
\\
A_{28}&=
\begin{pmatrix}
 1 & 1 & 1 & 0 & 1 & 1 \\
 0 & 1 & 0 & 1 & 0 & 1 \\
 0 & 0 & 1 & 1 & 0 & 1 \\
 0 & 0 & 0 & 1 & 0 & 0 \\
 0 & 0 & 0 & 0 & 1 & 1 \\
 0 & 0 & 0 & 0 & 0 & 1 \\
\end{pmatrix}
\quad
A_{29}=
\begin{pmatrix}
 1 & 1 & 0 & 1 & 1 & 1 \\
 0 & 1 & 0 & 1 & 0 & 0 \\
 0 & 0 & 1 & 1 & 1 & 0 \\
 0 & 0 & 0 & 1 & 0 & 0 \\
 0 & 0 & 0 & 0 & 1 & 1 \\
 0 & 0 & 0 & 0 & 0 & 1 \\
\end{pmatrix}
\quad
A_{30}=
\begin{pmatrix}
 1 & 0 & 1 & 1 & 1 & 1 \\
 0 & 1 & 0 & 0 & 0 & 1 \\
 0 & 0 & 1 & 0 & 1 & 1 \\
 0 & 0 & 0 & 1 & 0 & 0 \\
 0 & 0 & 0 & 0 & 1 & 0 \\
 0 & 0 & 0 & 0 & 0 & 1 \\
\end{pmatrix}
\quad
A_{31}=
\begin{pmatrix}
 1 & 1 & 1 & 1 & 1 & 1 \\
 0 & 1 & 0 & 1 & 0 & 1 \\
 0 & 0 & 1 & 1 & 0 & 0 \\
 0 & 0 & 0 & 1 & 0 & 0 \\
 0 & 0 & 0 & 0 & 1 & 1 \\
 0 & 0 & 0 & 0 & 0 & 1 \\
\end{pmatrix}
\end{align}

}

\vspace*{\fill}

The cycle graph reveals that $\Aut F_0(7)=G_{32}\!^6=\left((C_4\times C_2)\rtimes C_2\right)\rtimes C_2$.

\newpage

\vspace*{\fill}

\begin{figure}[H]
\centering
\includegraphics[scale=1.5]{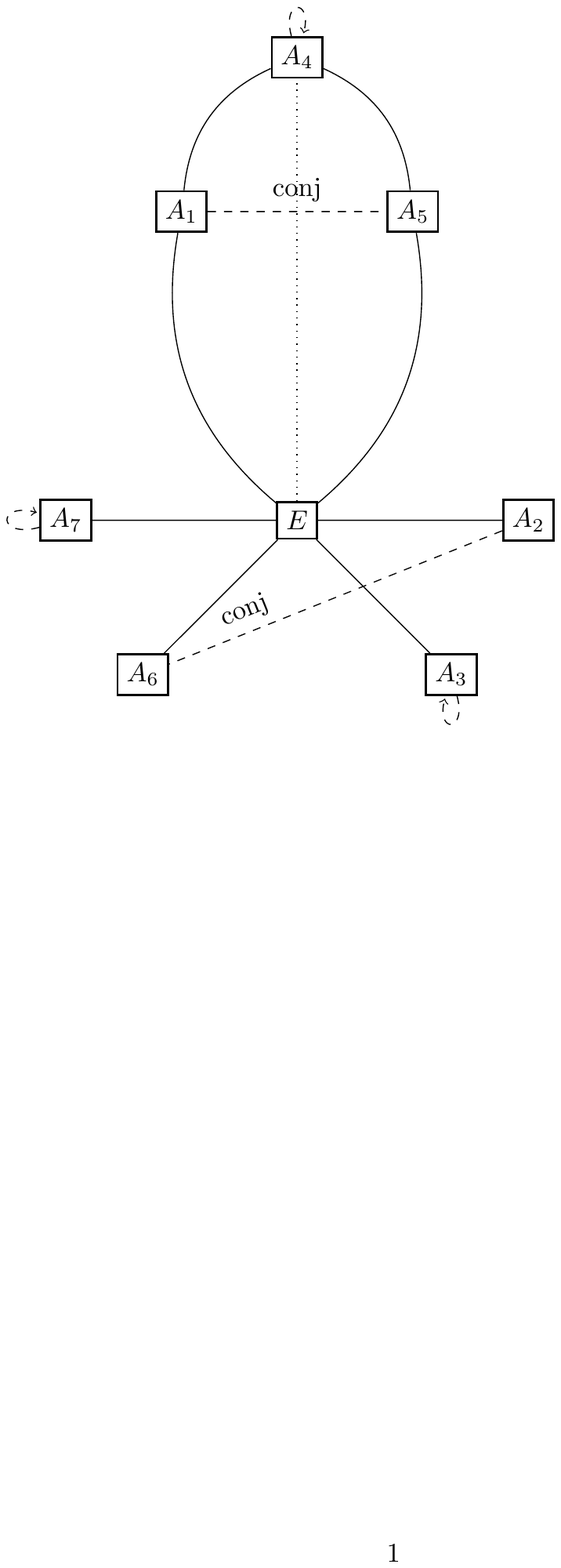}
\caption{\emph{Cycle graph for $\Aut F_0(5)=C_4\rtimes C_2=D_4$, the Dihedral Group of order 8,
with $C_2$ acting by inversion. The dotted lines indicate conjugation, as in} Definition~\ref{Conjugation}.
}
\end{figure}

\vspace*{\fill}

\newpage

\vspace*{\fill}

\begin{figure}[H]
\centering
\includegraphics[scale=1.2]{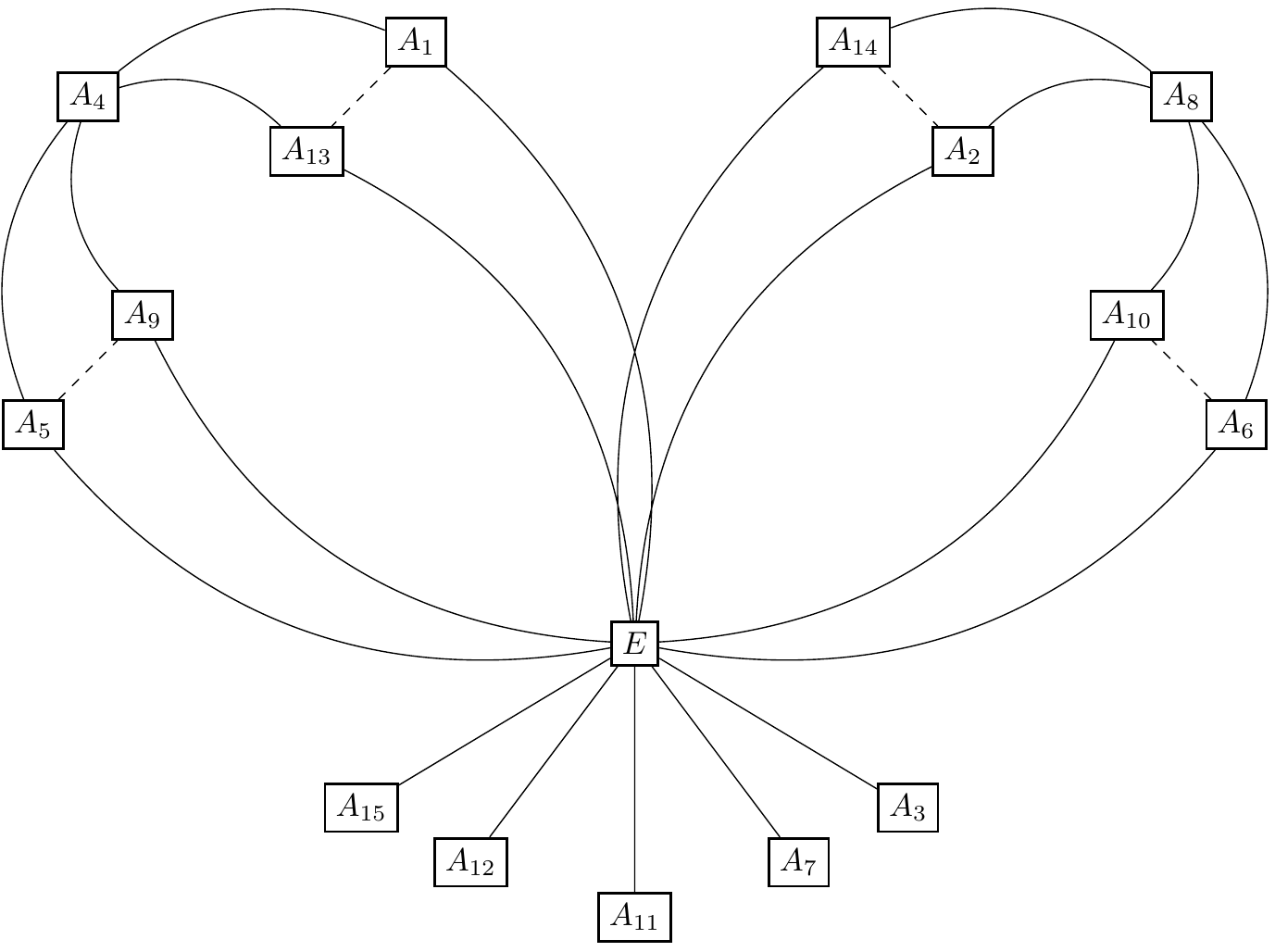}
\caption{\emph{Cycle graph for $\Aut F_0(6)=((C_4\times C_2)\rtimes C_2)$}.}
\end{figure}

\vspace*{\fill}


\newpage
\begin{thebibliography}{10}

\bibitem{abl2022}
R.~Ab{\l}amowicz,
\newblock {\it On the structure of ternary {Clifford} algebras and their
  irreducible representations},
\newblock Adv. Appl. Clifford Algebr.
\newblock {\bf  32} (2022), 39.

\bibitem{bag/lam1}
J.~Bagger and N.~Lambert,
\newblock {\it Gauge symmetry and supersymmetry of multiple {M2}-branes},
\newblock Phys. Rev.
\newblock {\bf  D77} (2008), 065008.

\bibitem{boh/wer2}
D.~Bohle and W.~Werner,
\newblock {\it A {K}-theoretic approach to the classification of symmetric
  spaces},
\newblock J. Pure and App. Algebra
\newblock {\bf  219} (2015), 4295--4321.

\bibitem{cel}
N.~Celakoski,
\newblock {\it On $({F},{G})$-rings},
\newblock God. Zb., Mat. Fak. Univ. Kiril Metodij Skopje
\newblock {\bf  28} (1977), 5--15.

\bibitem{cro1}
G.~Crombez,
\newblock {\it The {P}ost coset theorem for $(n,m)$-rings},
\newblock Ist. Veneto Sci. Lett. Arti, Atti, Cl. Sci. Mat. Natur.
\newblock {\bf  131} (1973), 1--7.

\bibitem{cro/tim}
G.~Crombez and J.~Timm,
\newblock {\it On $(n,m)$-quotient rings},
\newblock Abh. Math. Semin. Univ. Hamb.
\newblock {\bf  37} (1972), 200--203.

\bibitem{azc/izq}
J.~de$\;$Azcarraga and J.~M. Izquierdo,
\newblock {\it $n$-{A}ry algebras: {A} review with applications},
\newblock J. Phys.
\newblock {\bf  A43} (2010), 293001.

\bibitem{dor3}
W.~D\"ornte,
\newblock {\it Unterschungen \"uber einen verallgemeinerten {G}ruppenbegriff},
\newblock Math. Z.
\newblock {\bf  29} (1929), 1--19.

\bibitem{duplij2022}
S.~Duplij,
\newblock {\it Polyadic Algebraic Structures},
\newblock IOP Publishing,
\newblock Bristol, 2022.

\bibitem{dup/wer2021}
S.~Duplij and W.~Werner,
\newblock {\it Structure of unital 3-fields},
\newblock Math. Semesterber.
\newblock {\bf  68} (2021), 27--53.

\bibitem{elg/bre}
H.~A. Elgendy and M.~R. Bremner,
\newblock {\it Universal associative envelopes of {$(n+1)$}-dimensional
  {$n$}-{L}ie algebras},
\newblock Commun. Algebra
\newblock {\bf  40} (2012), 1827--1842.

\bibitem{hig}
P.~M. Higgins,
\newblock {\it Completely semisimple semigroups and epimorphisms},
\newblock Proc. Amer. Math. Soc.
\newblock {\bf  96} (1986), 387--390.

\bibitem{ker2000}
R.~Kerner,
\newblock {\it Ternary algebraic structures and their applications in physics},
\newblock in 23rd International Colloquium on Group Theoretical Methods in
  Physics,  (V.~Dobrev, A.~Inomata, G.~Pogosyan, L.~Mardoyan, and A.~Sisakyan,
  eds.),
\newblock Joint Inst. Nucl. Res., JINR Publishing, Dubna, 2000
\newblock (arXiv preprint: math-ph/0011023).

\bibitem{ker2012}
R.~Kerner,
\newblock {\it A ${Z}_{3}$ generalization of {P}auli's principle, quark algebra
  and the {L}orentz invariance},
\newblock in The Sixth International School on Field Theory and Gravitation,
  (J.~{Alves Rodrigues}, Waldyr, R.~{Kerner}, G.~O. {Pires}, and C.~{Pinheiro},
  eds.), Vol. 1483 of {\it AIP Conference Series},
\newblock 2012, pp.  144--168.

\bibitem{kur1}
A.~G. Kurosh,
\newblock {\it Multioperator rings and algebras},
\newblock Russian Math. Surveys
\newblock {\bf  24} (1969), 1--13.

\bibitem{lee/but}
J.~J. Leeson and A.~T. Butson,
\newblock {\it On the general theory of $(m,n)$ rings.},
\newblock Algebra Univers.
\newblock {\bf  11} (1980), 42--76.

\bibitem{nam0}
Y.~Nambu,
\newblock {\it Generalized {H}amiltonian dynamics},
\newblock Phys. Rev.
\newblock {\bf  7} (1973), 2405--2412.

\bibitem{pos}
E.~L. Post,
\newblock {\it Polyadic groups},
\newblock Trans. Amer. Math. Soc.
\newblock {\bf  48} (1940), 208--350.

\end{thebibliography}
\end{document}